\documentclass[a4paper,12pt]{article}

\usepackage{latexsym}
\usepackage{amsmath,amsthm,amssymb,amsfonts}
\usepackage{graphicx}
\usepackage{tikz}

\newcommand{\B}{\mathcal{B}}

\DeclareMathOperator{\Aut}{Aut}

\newtheorem{theorem}{Theorem}[section]
\newtheorem{proposition}[theorem]{Proposition}
\newtheorem{lemma}[theorem]{Lemma}
\newtheorem{corollary}[theorem]{Corollary}
\newtheorem{conjecture}[theorem]{Conjecture}
\newtheorem{remark}[theorem]{Remark}
\newtheorem{definition}[theorem]{Definition}
\newtheorem{problem}{Problem}[section]

\theoremstyle{definition}
\newtheorem{example}[theorem]{Example}

\theoremstyle{remark}

\newenvironment{Proof}{\noindent {\bf Proof.} \ }{\hfill $\Box$ \bigskip}

\begin{document}

\title{The Sierpi\'nski product of graphs }

\author{
  Jurij Kovi\v{c},   
  Toma\v{z} Pisanski, 
  Sara Sabrina Zemlji\v c and 
  Arjana \v{Z}itnik        
}
\date{April 1, 2019}

\maketitle

\begin{abstract}
In this paper we introduce a product-like operation that generalizes the construction of generalized Sierpi\'nski graphs.
Let $G,H$ be graphs and let $f: V(G) \to V(H)$ be a function. Then the {\em Sierpi\'nski product of $G$ and $H$ with respect to $f$} is defined as a pair $(K,\varphi)$, where $K$ is a graph on the vertex set  $V(G) \times V(H)$ with two types of edges:
\begin{itemize}
\item $\{(g,h),(g,h')\}$ is an edge in $K$ for every 
$g\in V(G)$ and every 
$\{h,h'\}\in E(H)$,
\item $\{(g,f(g'),(g',f(g))\}$ is an edge in $K$ for every edge $\{g,g'\} \in E(G)$; 
\end{itemize}  
and $\varphi: V(G) \to V(K)$ is a function that maps every vertex $g \in V(G)$ to the vertex $(g,f(g)) \in V(K)$. Graph $K$ will  be denoted  by $G \otimes_f H$. Function $\varphi$ is needed to define the product of more than two factors.
By applying this operation $n$ times to the same graph 
we obtain the $n$-th generalized Sierpi\'nski graph. 

Some basic properties of the Sierpi\'nski product are presented.
In particular, we show that $G \otimes_f H$ is connected if and only if
both $G$ and $H$ are connected and we present some necessary and
sufficient conditions that $G,H$ must fulfill in order for $G \otimes_f H$ to be planar.
As for symmetry properties, we show which automorphisms of $G$ and $H$
extend to automorphisms of $G \otimes_f H$. In many  cases we can
also describe the whole automorphism group of $G \otimes_f H$. 
\end{abstract}

\medskip

\noindent
{\sc Math. Subj. Class. (2010)}: 
05C76,  
05C12,  
05C10   

\section{Introduction} 

The family of \emph{Sierpi\'nski graphs} $S_p^n$ was first introduced 
by Klav\v zar and Milutinovi\'c in~\cite{KM1997} as a variant of the  Tower of Hanoi problem.
They can be defined recursively as follows:
$S_p^1$ is isomorphic to the complete graph $K_p$ and $S_p^{n+1}$ is constructed  from $p$ copies of $S_p^{n}$ by adding exactly one edge between every pair of copies of $S_p^{n+1}$.
Sierpi\'nski graphs $S_3^{1},S_3^{2}$, and $S_3^{3}$ are depicted in Figure~\ref{fig:triangle}.
In the ``classical'' case, when $p=3$, the Sierpi\'nski graphs are isomorphic to Hanoi graphs.
More about  Sierpi\'nski graphs and their connections to the Hanoi graphs can be found in the recent second edition of the book about the Tower of Hanoi puzzle
by Hinz et al.~\cite{HKMP2018}. 

Sierpi\'nski graphs have  been extensively studied in most graph-theoretical aspects 
as well as in other areas of mathematics and even psychology. 
Some notable papers are \cite{HKZ2013,Huang,Jakovac,KlaMiPetr,KZ2013,Parisse,Xue1,Xue2}.
An extensive summary of topics studied on and around Sierpi\'nski graphs is available in the survey paper~by Hinz, Klav\v{z}ar and Zemlji\v{c} \cite{HKZ2017}. 
In that paper the authors introduced {\em Sierpi\'nski-type graphs} as graphs that are derived from or lead to the Sierpi\'nski triangle fractal.

\begin{figure}[!htbp]
\centering
\begin{tikzpicture}[style=thick,scale=0.6,x=1.0cm,y=1.0cm]
	\def\vr{2.5pt/0.6}
	\draw (0,0)--(7,0)--(3.5,6.1)--cycle;
	\draw  [fill=white] (0,0) circle (\vr);
	\draw  [fill=white] (7,0) circle (\vr);
	\draw  [fill=white] (3.5,6.1) circle (\vr);
\end{tikzpicture}
\quad
\begin{tikzpicture}[style=thick,scale=0.6,x=1.0cm,y=1.0cm]
	\def\vr{2.5pt/0.6}
	\draw (0,0)--(7,0)--(3.5,6.1)--cycle;
	\draw  (2,3.5)--(5,3.5);
	\draw  (5.5,2.6)--(4,0);
	\draw  (3,0)--(1.5,2.6);
	\draw  [fill=white] (0,0) circle (\vr);
	\draw  [fill=white] (7,0) circle (\vr);
	\draw  [fill=white] (3.5,6.1) circle (\vr);
	\draw  [fill=white] (3,0) circle (\vr);
	\draw  [fill=white] (1.5,2.6) circle (\vr);
	\draw  [fill=white] (4,0) circle (\vr);
	\draw  [fill=white] (5.5,2.6) circle (\vr);
	\draw  [fill=white] (5,3.5) circle (\vr);
	\draw  [fill=white] (2,3.5) circle (\vr);
\end{tikzpicture}
\quad
\begin{tikzpicture}[style=thick,scale=0.6,x=1.0cm,y=1.0cm]
	\def\vr{2.5pt/0.6}
	\draw (0,0)--(7,0);
	\draw  (7,0)--(3.5,6.1);
	\draw  (3.5,6.1)--(0,0);
	\draw  (2,3.5)--(5,3.5);
	\draw  (5.5,2.6)--(4,0);
	\draw  (3,0)--(1.5,2.6);
	\draw  (0.5,0.9)--(1,0);
	\draw  (2,0)--(2.5,0.9);
	\draw  (2,1.7)--(1,1.7);
	\draw  (4.5,0.9)--(5,0);
	\draw  (6,0)--(6.5,0.9);
	\draw  (6,1.7)--(5,1.7);
	\draw  (2.5,4.3)--(3,3.5);
	\draw  (4,3.5)--(4.5,4.3);
	\draw  (4,5.2)--(3,5.2);
	\draw  [fill=white] (0,0) circle (\vr);
	\draw  [fill=white] (7,0) circle (\vr);
	\draw  [fill=white] (3.5,6.1) circle (\vr);
	\draw  [fill=white] (3,0) circle (\vr);
	\draw  [fill=white] (1.5,2.6) circle (\vr);
	\draw  [fill=white] (4,0) circle (\vr);
	\draw  [fill=white] (5.5,2.6) circle (\vr);
	\draw  [fill=white] (5,3.5) circle (\vr);
	\draw  [fill=white] (2,3.5) circle (\vr);
	\draw  [fill=white] (1,0) circle (\vr);
	\draw  [fill=white] (0.5,0.9) circle (\vr);
	\draw  [fill=white] (2,0) circle (\vr);
	\draw  [fill=white] (2.5,0.9) circle (\vr);
	\draw  [fill=white] (5,0) circle (\vr);
	\draw  [fill=white] (4.5,0.9) circle (\vr);
	\draw  [fill=white] (6,0) circle (\vr);
	\draw  [fill=white] (6.5,0.9) circle (\vr);
	\draw  [fill=white] (6,1.7) circle (\vr);
	\draw  [fill=white] (4,5.2) circle (\vr);
	\draw  [fill=white] (4.5,4.3) circle (\vr);
	\draw  [fill=white] (2,1.7) circle (\vr);
	\draw  [fill=white] (1,1.7) circle (\vr);
	\draw  [fill=white] (5,1.7) circle (\vr);
	\draw  [fill=white] (2.5,4.3) circle (\vr);
	\draw  [fill=white] (3,3.5) circle (\vr);
	\draw  [fill=white] (4,3.5) circle (\vr);
	\draw  [fill=white] (3,5.2) circle (\vr);
\end{tikzpicture}
\caption{Sierpi\'nski graphs $S_3^1$, $S_3^2$, and $S_3^3$.}
\label{fig:triangle}
\end{figure}
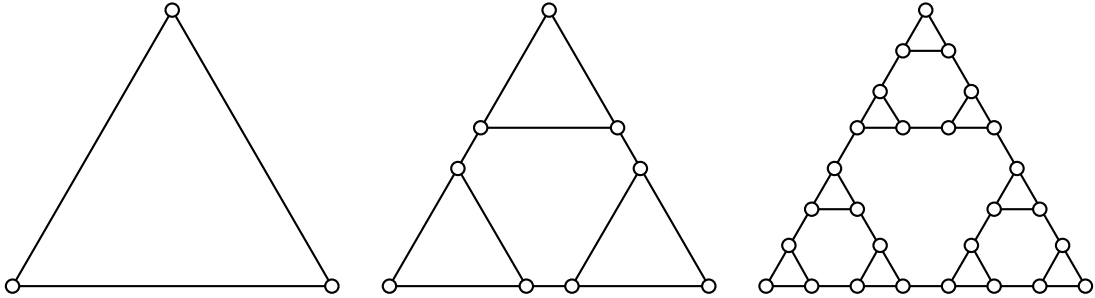

Recently these families of graphs have been generalized by Gravier, Kov\v se and Parreau
to a family called \emph{generalized Sierpi\'nski graphs}~\cite{Kovse}.
Instead of iterating a complete graph they start with an arbitrary graph $G$ and form a 
self-similar graph in the same way as Sierpi\'nski graphs are derived from a complete graph. See Figure~\ref{fig:house} for an example of the second iteration of a generalized Sierpi\'nski graph, where the base graph is a house. For a given graph $G$, $S_G^n$ denotes the $n$-th iteration generalized Sierpi\'nski graph. 

\begin{figure}[!htbp]
\centering
\begin{tikzpicture}[style = thick,x=1.0cm,y=1.0cm, scale=0.8]
	\def\vr{2.5pt/0.8}
	\draw (2,5.2)--(0,3.2)--(0,0)--(4,0)--(4,3.2)--cycle;
	\draw (0,3.2)--(4,3.2);
	\draw [fill=white] (2,5.2) circle (\vr);
	\draw [fill=white] (0,3.2) circle (\vr);
	\draw [fill=white] (0,0) circle (\vr);
	\draw [fill=white] (4,0) circle (\vr);
	\draw [fill=white] (4,3.2) circle (\vr);
\end{tikzpicture}
\qquad \qquad
\begin{tikzpicture}[style = thick,x=1.0cm,y=1.0cm, scale=0.8]
	\def\vr{2.5pt/0.8}
	\draw (0,2.4)--(0,0.8); 
	\draw (1,0)--(3,0); 
	\draw (4,0.8)--(4,2.4); 
	\draw (2.5,4.7)--(3.5,3.7); 
	\draw (1.5,4.7)--(0.5,3.7); 
	\draw (1,3.2)--(3,3.2); 
	\draw (1.5,4.7)--(1.5,3.9)--(2.5,3.9)--(2.5,4.7)--(2,5.2)--(1.5,4.7)--(2.5,4.7);
	\draw [fill=white] (2,5.2) circle (\vr);
	\draw [fill=white] (1.5,4.7) circle (\vr);
	\draw [fill=white] (1.5,3.9) circle (\vr);
	\draw [fill=white] (2.5,3.9) circle (\vr);
	\draw [fill=white] (2.5,4.7) circle (\vr);
	\draw (0,3.2)-- (0,2.4)--(1,2.4)--(1,3.2)--(0.5,3.7)--(0,3.2)--(1,3.2);
	\draw [fill=white] (0.5,3.7) circle (\vr);
	\draw [fill=white] (0,3.2) circle (\vr);
	\draw [fill=white] (0,2.4) circle (\vr);
	\draw [fill=white] (1,2.4) circle (\vr);
	\draw [fill=white] (1,3.2) circle (\vr);
	\draw (0,0.8)--(0,0)--(1,0)--(1,0.8)--(0.5,1.3)--(0,0.8)--(1,0.8);
	\draw [fill=white] (0.5,1.3) circle (\vr);
	\draw [fill=white] (0,0.8) circle (\vr);
	\draw [fill=white] (0,0) circle (\vr);
	\draw [fill=white] (1,0) circle (\vr);
	\draw [fill=white] (1,0.8) circle (\vr);
	\draw (3,0.8)--(3,0)--(4,0)--(4,0.8)--(3.5,1.3)--(3,0.8)--(4,0.8);
	\draw [fill=white] (3.5,1.3) circle (\vr);
	\draw [fill=white] (3,0.8) circle (\vr);
	\draw [fill=white] (3,0) circle (\vr);
	\draw [fill=white] (4,0) circle (\vr);
	\draw [fill=white] (4,0.8) circle (\vr);
	\draw (3,3.2)--(3,2.4)--(4,2.4)--(4,3.2)--(3.5,3.7)--(3,3.2)--(4,3.2);
	\draw [fill=white] (3.5,3.7) circle (\vr);
	\draw [fill=white] (3,3.2) circle (\vr);
	\draw [fill=white] (3,2.4) circle (\vr);
	\draw [fill=white] (4,2.4) circle (\vr);
	\draw [fill=white] (4,3.2) circle (\vr);
\end{tikzpicture}
\caption{Generalized Sierpi\'nski graphs $S_G^1$  and $S_G^2$ when $G$ is the house graph.}
\label{fig:house}
\end{figure}
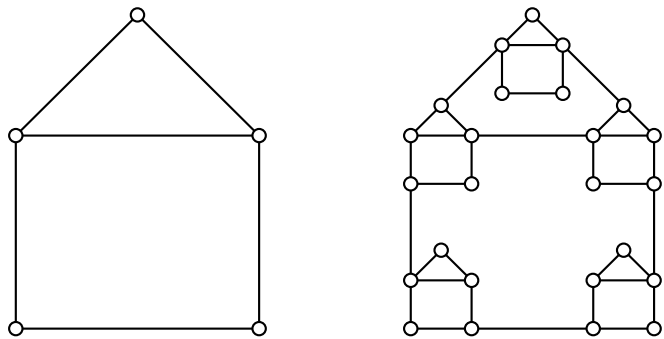



The generalized Sierpi\'nski graphs have been extensively studied in the past few years. 
A few years after they were introduced in 2011 the first two papers appeared at the same time. 
Geetha and Somasundaram~\cite{geetha-2015} studied their total chromatic number while Rodr\'{i}guez-Vel\'{a}zquez and Tom\'{a}s-Andreu~\cite{rodriguez-2015} examined their Randi\'c index. Shortly afterwards several papers followed on similar topics, but also on the chromatic number, vertex cover number, clique number, and domination number, see~\cite{rodriguez-2016+}. 

Metric properties have always presented intriguing problems in the family of Sierpi\'nski-type graphs 
mostly due to their connection to the Hanoi graphs. 
Namely a solution to the Tower of Hanoi problem may be modelled as a shortest path problem
on the corresponding Hanoi graph.
Therefore it is not surprising that metric properties of generalized Sierpi\'nski graphs have been studied as well. 
In ~\cite{estrada-moreno-2016b+} Estrada-Moreno, Rodr\'{i}guez-Bazan   and Rodr\'{i}guez-Vel\'{a}zquez 
investigate distances and present, among other results, an algorithm for determining the distance between an extreme vertex and an arbitrary vertex of a generalized Sierpi\'nski graph. 
In  the recent paper~\cite{AEKP}  Alizadeh et al.  investigate metric properties for generalized Sierpi\'nski graphs where the base graph is a star graph. 

At this point we would like to mention another approach towards the Sierpi\'nski graphs. 
The graphs $S_3^n$  appear naturally locally by applying a series of truncations of maps;
see Pisanski and Tucker \cite{PisanskiTucker} and Alspach and Dobson \cite{AlspachDobson}. 
For a cubic graph $G$ this is equivalent to applying a series of truncations to $G$,
where the truncation of $G$ is the line graph  of the subdivision graph of $G$. 
For any graph and the neighbourhood of vertex of valence $d$ the repeated truncation 
looks like $S_d^n$. 
A related construction, called the \emph{clone cover}, is considered 
by Malni\v{c}, Pisanski and \v{Z}itnik in  \cite{MPZ2015}.

In this paper we generalize the generalized Sierpi\'nski graphs even further. 
Instead of taking just one graph, we take two (or multiple)  graphs 
and present the operation that yields $S_G^n$ from $S_G^{n-1}$ and $G$ as a product. 
If we take two  graphs $G$ and $H$, the resulting product  locally has the structure of $H$, 
but globally it is similar to $G$. We call such a product operation the {\em Sierpi\'nski product}. 

The Sierpi\'nski product shows some features of classical graph products \cite{HIK-2011},
the most important being that the vertex set of the Sierpi\'nski product of graphs $G$ and $H$ is 
$V(G) \times V(H)$. However, one needs some extra information outside $G$ and $H$ to
complete the definition of the Sierpi\'nski product of graphs $G$ and $H$. This information
can be encoded as a function $f: V(G) \to V(H)$. 
Furthermore, the product is defined so that we can extend it to multiple factors.
We will see  that by definition the Sierpi\'nski product of two graphs is always
a subgraph of their lexicographic product. 

The paper is organized as follows. 
In Section 2 we give a formal definition of  the Sierpi\'nski product  of  graphs
$G$ and $H$ with respect to  $f: V(G) \to V(H)$, this product is denoted by $G \otimes_f H$.
We explore some  graph-theoretical properties such as connectedness and planarity
of the Sierpi\'nski product.
In particular, we show that $G \otimes_f H$ is connected if and only if
both $G$ and $H$ are connected and we present some necessary and
sufficient conditions that $G,H$ must fulfill in order for $G \otimes_f H$ to be planar.
In Section 3 we study symmetries of  the Sierpi\'nski product of two graphs.
We focus on the automorphisms of $G \otimes_f H$ 
that arise from 
the automorphisms of its factors
and study  the group, generated by these automorphisms. 
In many  cases we can also describe the whole automorphism group of $G \otimes_f H$. 
Finally in Section 4 we consider  the Sierpi\'nski product of more than two graphs. 
In the special case when we have $n$ equal factors, say equal to $G$, 
and $f:V(G) \to V(G)$ is the identity function, their Sierpi\'nski product
is equal to $S_G^n$.

\section{Definition of  the Sierpi\'nski product and basic properties}
\label{section:definitions}

Let us first review some necessary notions.
All the graphs we consider are undirected and simple.
Let $G$ be a graph and $x$ be a vertex of $G$.
By $N(x)$ we denote the set of vertices of $G$ that are adjacent to $x$, i.e., the neighborhood of $x$. 
Vertices in this paper will usually be tuples, but instead of writing them in vector form $(x_m,\ldots, x_1)$, we will usually write them as words $x_m\ldots x_1$. More precisely, vertices $(0,0,0)$ or $(0,(0,0))$ will simply be denoted by $000$, except in the case when we will emphasize their origins.
The number of vertices of a graph $G$, i.e., the order of $G$, will be denoted by $|G|$, and the number of edges of $G$, i.e., the size of $G$, will be denoted by $||G||$.

\begin{definition}
\label{def:sierpinski-product}
Let $G,H$ be graphs and let $f: V(G) \to V(H)$ be a function. Then the {\em Sierpi\'nski product of $G$ and $H$ with respect to $f$} is defined as a pair $(K,\varphi)$, where $K$ is a graph on the vertex set  $V(K)=V(G) \times V(H)$ with two types of edges:
\begin{itemize}
\item $\{(g,h),(g,h')\}$ is an edge in $K$ for every vertex $g\in V(G)$ and every edge $\{h,h'\}\in E(H)$,
\item $\{(g,f(g'),(g',f(g))\}$ is an edge in $K$ for every edge $\{g,g'\} \in E(G)$; 
\end{itemize}  
and $\varphi: V(G) \to V(K)$ is a function that maps every vertex $g \in V(G)$ to the vertex $(g,f(g)) \in V(K)$. We will denote such Sierpi\'nski product by $G \otimes_f H$. 
\end{definition}

Often when we will have only two factors, we will  be interested only in the graph $K$ and not the embedding $\varphi$ of $G$ into $K$. 
In such cases we will use the notation $K = G \otimes_f H$.
If $V(G) \subseteq V(H)$ and $f$ is the identity function on its domain, we will skip the index $f$ 
and denote the Sierpi\'nski product of $G$ and $H$ simply by $G \otimes H$. 
The role of function  $\varphi$ will become clear in Section \ref{section:morefactors}. 
Note that there are no restrictions on function $f:V(G) \to V(H)$.
However, sometimes it is convenient that for every $g,g_1,g_2 \in V(G)$
the following property holds:  if $g_1,g_2 \in N(g)$, then $f(g_1) \ne f(g_2)$. 
In this case we say that $f$ is {\em locally injective}.

The Sierpi\'nski product can be defined in a similar way also for graphs with loops and multiple edges.
In this case, a loop in $G$, say $\{g,g\}$, would correspond to a loop $\{(g,f(g)),(g,f(g))\}$ in $G\otimes_f H$ and a multiple edge in $G$ would correspond to a multiple edge in $G\otimes_f H$,
but all our graphs will 	be simple. 

Figure \ref{fig:K4_C3}, left, shows the Sierpi\'nski product of $C_3$ and $K_4$ 
with respect to function $f_1$. Vertices of $C_3$ are labeled with numbers  $0,1,2$, 
vertices of $K_4$ are labeled with numbers  $0,1,2,3$ 
and  $f_1:V(C_3) \to V(K_4)$ is the identity function on its domain.
Figure \ref{fig:K4_C3}, right, shows the Sierpi\'nski product of  $K_4$ and $C_3$ with
respect to $f_2:V(K_4) \to V(C_3)$ defined as
$f_2(4) = 3$ and $f_2(i) = i$ otherwise.
This shows that the Sierpi\'nski product is not commutative.

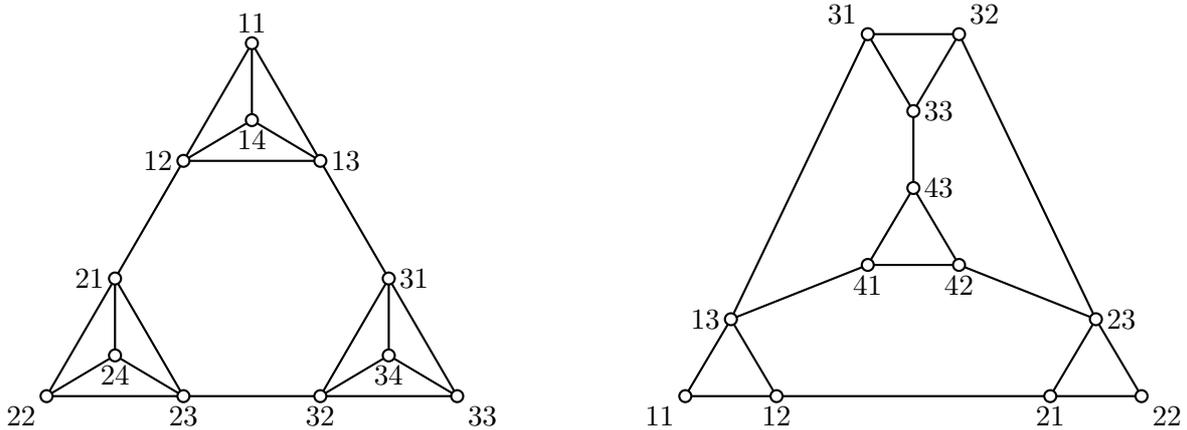
\begin{figure}[htb]
	\begin{center}
		\begin{tikzpicture}[scale=0.6,style=thick]
		\begin{small}
		\def\vr{3pt/0.8}
			\draw (3,0)--(6,0);
			\draw (7.5,2.6)--(6,5.2);
			\draw (3,5.2)--(1.5,2.6);
			\draw (3,5.2)--(6,5.2)--(4.5,7.8)--cycle;
			\draw (3,5.2)--(4.5,6.1)--(4.5,7.8);
			\draw (4.5,6.1)--(6,5.2);
			\draw [fill=white] (3,5.2) circle (\vr) node [anchor = east] {12};
			\draw [fill=white] (6,5.2) circle (\vr) node [anchor = west] {13};
			\draw [fill=white] (4.5,7.8) circle (\vr) node [anchor = south] {11};
			\draw [fill=white] (4.5,6.1) circle (\vr)node [anchor = north] {14};
			\draw (0,0)--(3,0)--(1.5,2.6)--cycle;
			\draw (1.5,2.6)--(1.5,0.9)--(0,0);
			\draw (1.5,0.9)--(3,0);
			\draw [fill=white] (0,0) circle (\vr) node [anchor = north east] {22};
			\draw [fill=white] (3,0) circle (\vr) node [anchor = north] {23};
			\draw [fill=white] (1.5,2.6) circle (\vr) node [anchor = east] {21};
			\draw [fill=white] (1.5,0.9) circle (\vr) node [anchor = north] {24};
			\draw (6,0)--(9,0)--(7.5,2.6)--cycle;
			\draw (9,0)--(7.5,0.9)--(6,0);
			\draw (7.5,2.6)--(7.5,0.9);
			\draw [fill=white] (6,0) circle (\vr) node [anchor = north] {32};
			\draw [fill=white] (9,0) circle (\vr) node [anchor = north west] {33};
			\draw [fill=white] (7.5,2.6) circle (\vr) node [anchor = west] {31};
			\draw [fill=white] (7.5,0.9) circle (\vr) node [anchor = north] {34};
			\draw (16,0)--(22,0);
			\draw (18,8)--(15,1.7);
			\draw (20,8)--(23,1.7);
			\draw (19,4.6)--(19,6.3);
			\draw (18,2.9)--(15,1.7);
			\draw (20,2.9)--(23,1.7);
			\draw (14,0)--(16,0)--(15,1.7)--cycle;
			\draw [fill=white] (14,0) circle (\vr) node [anchor = north east] {11};
			\draw [fill=white] (16,0) circle (\vr) node [anchor = north] {12};
			\draw [fill=white] (15,1.7) circle (\vr) node [anchor = east] {13};
			\draw (22,0)--(24,0)--(23,1.7)--cycle;
			\draw [fill=white] (22,0) circle (\vr) node [anchor = north] {21};
			\draw [fill=white] (24,0) circle (\vr) node [anchor = north west] {22};
			\draw [fill=white] (23,1.7) circle (\vr) node [anchor = west] {23};
			\draw (20,8)--(18,8)--(19,6.3)--cycle;
			\draw [fill=white] (18,8) circle (\vr) node [anchor = south east] {31};
			\draw [fill=white] (20,8) circle (\vr) node [anchor = south west] {32};
			\draw [fill=white] (19,6.3) circle (\vr) node [anchor = west] {33};
			\draw (18,2.9)--(20,2.9)--(19,4.6)--cycle;
			\draw [fill=white] (18,2.9) circle (\vr) node [anchor = north] {41};
			\draw [fill=white] (20,2.9) circle (\vr) node [anchor = north] {42};
			\draw [fill=white] (19,4.6) circle (\vr) node [anchor = west] {43};
		\end{small}
		\end{tikzpicture}
	\end{center}
\caption{Graphs $K_3 \otimes K_4$ and $K_4 \otimes_{f_2} K_3$, where $f_2(4) = 3$ and $f_2(i) = i$ otherwise.}
\label{fig:K4_C3}
\end{figure}

We now state some simple lemmas regarding the structure of the Sierpi\'nski product of  two graphs.
We omit most of the proofs, since they follow straight from the definition.

\begin{lemma}
\label{lem:trivial}
Let $G,H$ be graphs  and let $f: V(G) \to V(H)$ be a function.
Then the following statements hold.
\begin{itemize}
\item[{\rm (i)}] If $|G|=1$, then $G\otimes_f H$ is isomorphic to $H$.
\item[{\rm (ii)}] If $|H|=1$, then $G\otimes_f H$ is isomorphic to $G$.
\end{itemize}
\end{lemma}

\begin{lemma}
\label{lem:subgraph}
Let $G,H$ be graphs  and let $f: V(G) \to V(H)$ be a function.
Let $G',H'$ be  subgraphs of $G,H$, respectively, and 
let $f'$ be the  restriction od $f$ to $V(G')$ such that ${\rm Im}(f') \subseteq V(H')$.
Then $G'\otimes_{f'} H'$ is a subgraph of $G\otimes_f H$.
\end{lemma}

\begin{lemma}
\label{lem:HGcopy}
Let $G,H$ be graphs and let $f: V(G) \to V(H)$ be a function.
Then the following statements hold.
\begin{itemize}
\item[{\rm (i)}] Let $g$ be a vertex of $G$. Then the subgraph of $G \otimes_f H$, 
induced by the set $\{(g,h)| \, h \in V(H)\}$ is isomorphic to $H$.
\item[{\rm (ii)}] Graph $G$ is a minor of $G \otimes_f H$.
\end{itemize}
\end{lemma}

We say that the subgraph of $G \otimes_f H$ from Lemma \ref{lem:HGcopy} (i) 
is \emph{associated with $g$} and denote it by $gH$. 
We may view $G \otimes_f H$ as obtained from identical copies of $H$, one for each vertex of $G$, 
and attaching  for every edge $\{g,g'\} \in E(G)$ the  corresponding vertex $f(g)$ in $g'H$ 
to the vertex $f(g')$ in $gH$. 
The edges of $G \otimes_f H$ naturally fall into two classes.  
{All edges connecting different copies $gH$ are called {\em connecting edges}, 
while the edges inside some subgraph $gH$ are called {\em inner edges}.}

\begin{lemma}
\label{lem:edges}
Let $G,H$ be graphs and let $f: V(G) \to V(H)$ be a function.
Then $G\otimes_f H$ has $|G| \cdot |H|$ vertices and $||H||\cdot |G|+||G||$ edges.
In particular, $G\otimes_f H$ has $||H||\cdot |G|$ inner edges and
$||G||$ connecting edges.
\end{lemma}

\begin{lemma}
	\label{lem:atmostone}
Let $G$ and $H$ be  graphs and let $f: V(G)  \to V(H)$ be any mapping.
Then the following holds.
\begin{itemize}
\item[{\rm (i)}] There is at most one edge connecting  $gH$ and $g'H$ for every $g,g' \in V(G)$. 
\item[{\rm (ii)}] Suppose that $f$ is locally injective. 
Then any vertex of $G \otimes_f H$ is an end-vertex of at most one connecting edge.
\end{itemize}
\end{lemma}
\begin{proof}
(i) Only $\{(g,f(g')),(g',f(g))\}$ can connect $gH$ and $g'H$ and since $G$ is simple there is only one 	such edge.

\noindent
(ii) Let $(g,h)$ be a vertex of   $G \otimes_f H$.
Since $f$ is locally injective, there exists at most one vertex $g' \in N(g)$ such that
$h = f(g')$. If such a vertex exists, then $(g,h)$ is an end-vertex
of the edge $\{(g,h),(g',f(g)\}$, otherwise it is not contained in any connecting edge.
\end{proof}

The lexicographic product  of graphs $G$ and $H$ is a graph $G \circ H$ with vertex set
$V(G) \times V(H)$ and   two vertices $(g,h)$ and $(g',h')$ are adjacent in $G \circ H$
if and only if either $g$ is adjacent with $g'$ in $G$ or $g=g'$ and $h$ is adjacent with $h'$ in $H$.
In other words, $G \circ H$ consists of $|G|$ copies of $H$ and for every edge
$\{g,g'\}$ in $G$, every vertex of $gH$ is connected to every vertex in $g'H$.
Therefore the next result follows straight from 	Definition~\ref{def:sierpinski-product}.

\begin{proposition}  \label{prop:lexicographic}
Let $G$ and $H$ be  graphs and let $f: V(G)  \to V(H)$ be any mapping.
Then $G \otimes_f H$ is a subgraph of $G \circ H$.
\end{proposition}

Note that for different functions $f,f'$ graphs 
$G\otimes_f H$ and $G\otimes_{f'} H$ may be isomorphic or nonisomorphic.

\begin{theorem}   \label{thm:isomorphicGxH}
Let $G,H$ be graphs and let $f: V(G) \to V(H)$ be a function.
Let $\alpha \in \Aut(G)$, $\beta \in \Aut(H)$ and $f'=\beta \circ f \circ \alpha$.
Then $G\otimes_f H$ is isomorphic to $G\otimes_{f'} H$.
\end{theorem}
\begin{proof}
Define a function $\gamma: V(G\otimes_f H) \to V(G\otimes_f'H)$ by
$\gamma(g,h) = (\alpha^{-1}(g),\beta(h))$ for $g \in V(G)$ and $h \in V(h)$.
Since $\alpha,\beta$ are bijections, also $\gamma$ is a bijection.
Since $\beta$ is an automorphism, $\gamma$ maps inner edges to inner edges.

Take a connecting edge in $G\otimes_f H$, say $\{(g,f(g')),(g',f(g))\}$,
where $\{g,g'\} \in E(G)$. Then $\gamma(g,f(g'))=(\alpha^{-1}(g),\beta(f(g'))$
and $\gamma(g',f(g))=(\alpha^{-1}(g'),\beta(f(g))$.
Since $f'(\alpha^{-1}(g))=\beta(f(\alpha(\alpha^{-1}(g))))=\beta(f(g))$ and
$f'(\alpha^{-1}(g'))=\beta(f(\alpha(\alpha^{-1}(g'))))=\beta(f(g'))$,
we see that $\gamma$ also maps a connecting edge to a connecting edge.
Therefore $\gamma$ is an isomorphism.
\end{proof}

\begin{corollary}  \label{cor:iso}
Let $G$ be a graph and let $f \in \Aut(G)$. Then
$G\otimes G$ is isomorphic to $G\otimes_f G$.
\end{corollary}

In the remainder of this section we consider two other basic  graph-theoretic properties
of the Sierpi\'nski product with respect to its factors: connectedness and planarity.

\begin{proposition}
\label{prop:connected}
Let $G$ and $H$ be graphs and let $f: V(G) \to V(H)$ be a function.
Then $G\otimes_f H$ is connected    if and only if $G$ and $H$ are connected.
\end{proposition}

\begin{Proof} 
Suppose $G$ and $H$ are  connected. Pick two vertices $(g,h), (g',h')$ in $G\otimes_f H$.
Then there exists a path from $g$ to $g'$ in $G$, say $g=g_0,g_1,g_2,\dots,g_{k-1},g_k=g'$.
We construct a path from $(g,h)$ to $(g',h')$ in $G\otimes_f H$, passing
through $g_1H,g_2H,\dots,g_{k-1}H$ in that order as follows.
Let $P_0$ be a path  from $(g_0,h)$ to $(g_0,f(g_1))$ in $g_0H$,
let $P_i$ be a path from $(g_i,f(g_{i-1}))$ to $(g_i,f(g_{i+1}))$ in $g_iH$ for $i=1, \dots k-1$ and 
let $P_k$ be a path  from $(g_k,f(g_{k-1}))$ to $(g_k,h')$ in $g_kH$.
Such  paths exist since every subgraph $g_iH$ is connected.
Then $P_0P_1\dots P_k$ is a path between $(g,h), (g',h')$ in $G\otimes_f H$.

Conversely, suppose $G\otimes_f H$ is connected. 
Pick two vertices $g$ and  $g'$ from $G$.
Then a path from $gH$ to $g'H$ in $G\otimes_f H$  
corresponds to a path in $G$ from $g$ to $g'$. Therefore also $G$ is connected.
Suppose now that $H$ is not connected. We will show that in this
case $G\otimes_f H$ is not connected.
Denote by $H_1$ a connected component of $H$ such that
$V_1=\{g \in G| \ f(g) \in V(H_1)\}$ is nonempty.
For $g \in V_1$ and $h \in V(H_1)$,  all the neighbours of $(g,h)$ 
belong to $gH_1$ or to $g'H_1$ for some $g' \in V_1$. 
Therefore there are no edges between the set of vertices
$\{(g,h) \in G\otimes_f H| \ g \in V_1\ \mbox{and} \ h \in V(H_1)\}$
and the rest of the vertices of  $G\otimes_f H$.
So  $G\otimes_f H$ is not connected. This finishes the proof.
\end{Proof}

We will denote by $H+g$ the graph obtained from $H$ by adding a copy  of  vertex $g \in V(G)$ to it 
and connecting it to all vertices $f(g')$, where  $g' \in N(g)$. We  will  denote this new vertex by $g_H$.

The next Theorem characterises when a Sierpi\'nski product $G\otimes_f H$ is planar.

\begin{theorem}
	\label{thm:planar}
Let $G,H$ be connected graphs and let $f: V(G)  \to V(H)$ be any mapping.
Then $G\otimes_f H$ is planar if and only the following three conditions are fulfilled:
\begin{itemize}
\item[{\rm (i)}]  graph $G$ is planar,
\item[{\rm (ii)}] for every $g \in V(G)$ the graph $H+g$ is planar,
\item[{\rm (iii)}] there exists an embedding of $G$ in the plane with the following property:
                      for every $g \in V(G)$,  with $g_1,g_2,\dots, g_k$ being the cyclic order of vertices around $g$,
                       there exists an embedding of $H+g$ in the plane such that the cyclic order of vertices 
                       around $g_H$  in $H+g$ is $f(g_k),f(g_{k-1}),\dots,f(g_1)$.
\end{itemize}
\end{theorem}

\begin{Proof} All three conditions are necessary. 
Suppose  $G\otimes_f H$ is embedded in the plane.
Then a planar embedding of $G$ is obtained by contracting $gH$ 
to a single vertex for every $g \in V(G)$. Hence $G$ must be planar. 
Suppose $G$ is embedded in the plane as above. Let $g \in V(G)$ and let 
$g_1,g_2,\dots, g_k$,  be the cyclic order of vertices around $g$ in this embedding.
Let $N(gH)=\{g'H| g' \in N(g)\}$ denote the collection 
of graphs $g'H$  that are adjacent to $gH$ for some $g \in V(g)$. 
We contract every member $g'H$ from $N(gH)$ to a single vertex.
Then we keep $gH$, all the new vertices all the new edges  and delete the rest of the graph. 
The graph obtained in this way is still embedded in the plane.
Now we identify all the new vertices; 
we call the vertex obtained in this way $g_H$ for convenience. 
We obtain a graph that is isomorphic to $H+g$. The obtained graph $gH+g$ is planar. 
Moreover, the cyclic order of vertices around $g_H$ in $gH+g$ is  
$(g,f(g_k)),(g,f(g_{k-1})),\dots,(g,f(g_1))$.
Therefore the embedding of $G$, obtained from $G\otimes_f H$ by
contracting every copy of $H$, fulfills (iii).

The converse goes by construction. 
We first embed $G$ in the plane as in (iii) and then expand every vertex $g$ of $G$ to 
$gH$, embedded in the plane as in (iii). By (iii) it is possible to connect the copies of $H$ such that
the resulting graph is a plane embedding of $G\otimes_f H$. 
\end{Proof}

Next result follows directly from Theorem \ref{thm:planar} (ii).

\begin{corollary}
	\label{cor:sameface}
Let $G,H$ be connected graphs and let $f: V(G)  \to V(H)$ be any mapping.
If $G\otimes_f H$ is planar, then for every $g \in G$ there exists an embedding
of $H$ in the plane such that the vertices $\{f(g'); \ g' \in N(g)\}$ lie on the
boundary of the same face.
\end{corollary}

Using Theorem \ref{thm:planar} and Corollary \ref{cor:sameface} we  can
determine when $G\otimes G$ is planar for a connected graph $G$.
We also give a sufficient condition for $G\otimes_f H$ to be planar when
$G \ne H$.

\begin{corollary}
	\label{cor:planarequal}
Let $G$ be a connected graph and let $f: V(G)  \to V(G)$ be the identity mapping.
Then $G\otimes G$ is planar if and only if $G$ is outerplanar or $G=K_4$.
\end{corollary}
\begin{Proof} 
If $G$ is outerplanar or $K_4$, then conditions (i), (ii), (iii) from Theorem \ref{thm:planar}
are fulfilled, so $G\otimes G$ is planar.

Suppose now that $G$ is planar but  not outerplanar. Then it contains a subdivision
of $K_{2,3}$  or a subdivision of $K_4$ (with at least one additional vertex)  as a subgraph. 
Such a graph $G$ always contains a vertex such that in every plane embedding of $G$ 
not all of its neighbours will  be on the boundary of the same face. 
Therefore  $G\otimes G$ is not planar by Corollary \ref{cor:sameface}.
\end{Proof} 

\begin{theorem}
	\label{thm:planarlowdegree}
Let $G,H$ be connected graphs and let $f: V(G)  \to V(H)$ be any mapping.
Assume that  $G$ is planar, $\Delta(G) \le 3$ and $H$ is outerplanar. Then $G\otimes_f H$ is planar.
\end{theorem}
\begin{Proof} Denote $K=G\otimes_f H$.
Suppose $K$ is not planar. Then it contains
a subdivision of $K_{3,3}$ or $K_5$ as a subgraph. 
First assume that $K$ contains a subdivision of $K_{3,3}$. There are four cases to consider,
depending on how many vertices of degree $3$ of the subdivision of $K_{3,3}$ are in the
same copy of $H$.
\begin{enumerate}
\item If every vertex  is in separate copy of $H$ in $K$, then by contracting $gH$ to a single vertex 
          for every $g \in G$ we see that  $K_{3,3}$ is a minor in $G$, so $G$ is not planar.
\item  If there are between two and four vertices  in some $gH$,
        then we need at least four edges connecting $gH$ to other copies of $H$ in $K$.
        This is not possible, since maximal degree in $G$ is at most three.
\item There are five vertices  in  some $gH$ and one vertex  in some $g'H$ 
         for $g \ne g'$.   Since $H$ is outerplanar, $gH$ cannot contain a subdivision of $K_{2,3}$.
         Therefore we need  at least  two edges going out of $gH$ to obtain a subdivision
         of $K_{2,3}$ from the five vertices in $gH$. We also need three edges going out
          of $gH$ to connect $gH$ to the vertex of $K_{3,3}$ in $g'H$. This is again not
         possible, since the maximal degree of $G$ is at most 3.
\item The only remaining possibility is that all six vertices  are in the same copy $gH$ of $H$.
        Since $H$ is outerplanar, there can be at most seven edges (or paths) between pairs of vertices
        of $K_{3,3}$ in $gH$.  The remaining two paths must go through other copies of $H$, 
        which means that we again need   at least four edges connecting $gH$ to other copies of $H$ in $K$. 
         A contradiction.
\end{enumerate}

Therefore $K$ does not contain a subdivision of $K_{3,3}$.
Next assume that $K$ contains a subdivision of $K_{5}$. There are three cases to consider,
depending on how many vertices of degree $4$ of the subdivision of $K_{5}$ are in the
same copy of $H$.
\begin{enumerate}
\item If every vertex is in separate copy of $H$ in $K$, then by contracting $gH$ to a single vertex 
          for every $g \in G$ we see that  $K_{5}$ is a minor in $G$, so $G$ is not planar.
\item  If there are between two and four vertices  in some $gH$,
        then we need at least four edges connecting $gH$ to other copies of $H$ in $K$.
        This is not possible, since maximal degree in $G$ is at most three.
\item  The only remaining possibility is that all five vertices of $K_{5}$ are in the same copy of $H$.
        Since $H$ is outerplanar, it doesn't contain a subdivision of $K_4$ or $K_{2,3}$. 
        Therefore there can be  at most eight edges (or paths) between pairs of these vertices 
        in $gH$ (in fact, there can be at most six such paths).
        The remaining two paths   must go through other copies of $H$, which means that we  need
        at least four edges connecting $gH$ to other copies of $H$ in $K$. A contradiction.
\end{enumerate}
It follows that  $K$ doesn't contain a subdivision of $K_{3,3}$ or $K_5$, so it is planar.

\end{Proof} 

If a connected graph is not planar it is natural to consider its genus. The genus of a graph
$G$ is denoted by $\gamma(G)$. Recall that by  Lemma \ref{lem:HGcopy}, graph
$G$ is a minor of $G \otimes_f H$ for any function $f:V(G) \to V(H)$, 
and $G \otimes_f H$ contains $|G|$ copies of $H$ as induced subgraphs.
 Suppose $G,H$ are connected and $f$ is arbitrary. Then
it is easy to see, cf. \cite[Theorem 4.4.2]{Mohar}, that
\begin{equation}   \label{eq:genus}
\gamma(G \otimes_f H) \ge \gamma(G) + |G|\cdot \gamma(H).
\end{equation}
Note that the bound is not sharp even if the factors are planar.
In the case of planar Sierpi\'nski product we were able to settle the case in
Theorem  \ref{thm:planar}.
It would be interesting to find some sufficient condition for the equality
in \eqref{eq:genus} to hold also for non-planar Sierpi\'nski products.


\section{Symmetry}
\label{section:symmetry}

Throughout this section let $G,H$ be connected graphs and let $f: V(G) \to V(H)$ be  any mapping. 
Recall that  the edge set of  $G \otimes_f H$ can be naturally partitioned into two subsets:
\begin{itemize} 
\item \emph{inner edges} $\{(g,h),(g,h')\}$  for every vertex $g\in V(G)$ and every edge $\{h,h'\}\in E(H)$, and
\item  \emph{connecting edges} $\{(g,f(g'),(g',f(g))\}$ for every edge $\{g,g'\} \in E(G)$. 
\end{itemize}  
We call this partition of the edge set the \emph{fundamental edge partition}.
We will say that an automorphism of $G \otimes_f H$ \emph{respects the fundamental
edge partition} if it takes inner edges to inner edges, and connecting edges  to connecting edges. 
We denote the set of all automorphisms of $G \otimes_f H$ that respect the fundamental edge partition 
by  $\tilde{A}(G,H,f)$. It is easy to see that this set is a subgroup of the whole automorphism
group of $G \otimes_f H$. If $G,H$ are connected graphs, the automorphisms that respect the fundamental edge partition have the following useful property.

\begin{proposition}
Let $G$ and $H$ be connected  graphs. Then every 
automorphism $\tilde{\gamma} \in  \tilde{A}(G,H,f)$
permutes  the subgraphs $gH$, $g \in G$.
In particular, the restriction $\tilde{\gamma}|_{V(gH)}:V(gH) \to V(g'H)$,
where $g' \in V(G)$, is a graph isomorphism.
\end{proposition}

In this section we first show that any automorphism of $G \otimes_f H$
that respects the fundamental edge partition induces automorphisms of $G$ and $H$. And conversely,
we define two families of automorphisms of  $G \otimes_f H$
that respect the fundamental edge partition using automorphisms of $G$ and $H$. 
Then we show that  in many cases all the automorphisms 
of $G \otimes_f H$ respect the fundamental edge partition.
Finally, we focus on the case when $G=H$ and $f$ is an automorphism. 
In this case we can completely describe the group of automorphisms 
that respect the fundamental edge partition.


\subsection{Automorphisms that respect the fundamental edge partition}
\label{subsec:fep}

Let $\tilde\gamma$ be an automorphism  of $G \otimes_f H$ 
that respects the fundamental edge partition. Then it permutes
the subgraphs $gH$, $g \in G$.
Define a mapping $\gamma$ such that $\gamma(g)=g^\prime$
 if $\tilde\gamma$ maps $gH$ to $g^\prime H$.
Obviously,  $\gamma$ is a bijection. Let $\{g,g_1\}$ be 
an edge of $G$. Then $\{(g,f(g_1)),(g_1,f(g))\}$
is a connecting edge of $G \otimes_f H$. Since $\tilde\gamma$ respects
the fundamental edge partition, it maps this edge to another 
connecting edge, say $\{(g^\prime,f(g_1^\prime)),(g_1^\prime,f(g^\prime))\}$,
where $g^\prime$ and $g_1^\prime$ are adjacent in $G$. But then $\gamma$ maps the edge 
$\{g,g_1\}$ to an edge (i.e. to $\{g^\prime,g_1^\prime\}$) and  $\gamma$ is an automorphism.
We will say that $\gamma$ is the \emph{projection} of $\tilde\gamma$ on $G$.
Conversely, $\tilde\gamma$ is a \emph{lift} of $\gamma$. Note that projection of 
$\tilde\gamma \in \Aut(G \otimes_f H)$ on $G$ is uniquely defined. However, given an automorphism
of $G$, it can  have a unique lift,  more than one lift or none at all.

On the other hand, the action of $\tilde\gamma$ on every copy of $gH$ in $G \otimes_f H$
induces an automorphism $\gamma_g$ of $H$, defined by
$\gamma_g(h)=h^\prime$ if $\tilde\gamma$ sends $(g,h)$ to $(g_1,h^\prime)$
for some $g_1 \in V(G)$ and $h^\prime \in V(H)$.

We will now introduce two families of automorphism of $G \otimes_f H$
that can be obtained from automorphisms of $G$ and $H$. All such automorphisms
respect the fundamental edge partition.  

\begin{definition}   \label{def:auto1}
Let $G,H$ be connected graphs and let $f:V(G) \to V(H)$ be any function. 
Let $\alpha \in \Aut(G)$ and let $\B: V(G) \to \Aut(H)$ be any mappng. For
simplicity we will write $\beta_g$ instead of $\B(g)$ for $g \in V(g)$.
Define a mapping $\Psi(\alpha,\B):V(G \otimes_f H) \to V(G \otimes_f H)$  by
$$
\Psi(\alpha,\B): (g,h) \mapsto (\alpha(g),\beta_g(h)).
$$
If $\B$ is a constant function, say $\beta_g = \beta$ for all $g \in V(G)$, 
we denote $\Psi(\alpha,\B)$ by $\Psi(\alpha,\beta)$.
\end{definition}


By the discussion at the beginning of this section, we see that the following holds.

\begin{theorem}    \label{thm:autoform}
Let $G,H$ be connected graphs and let $f:V(G) \to V(H)$ any function. 
Every automorphism of $G \otimes_f H$ that respects the fundamental edge partition  
is of form $\Psi(\alpha,\B)$ for some $\alpha \in \Aut(G)$ and some 
mapping $\B:V(G) \to \Aut(H)$.
\end{theorem}

We now determine when the mapping $\Psi(\alpha,\B)$ from Definition  \ref{def:auto1} is an automorphism.

\begin{proposition}
\label{prop:auto11}
The mapping $\Psi(\alpha,\B)$ is always a bijection.
\end{proposition}
\begin{proof}
It is enough to prove that $\Psi(\alpha,\B)$ is injective.
This is straightforward  since $\alpha$ and $\beta_g$, $g \in V(G)$,
are all injective.
\end{proof}

\begin{proposition}
\label{prop:auto12}
The mapping $\Psi(\alpha,\B)$ is an automorphism if and only if
for every $g \in V(G)$ we have $f \circ \alpha = \beta_g \circ f$ on $N(g)$.
Moreover, in this case $\Psi(\alpha,\B)$ respects the fundamental edge partition.
\end{proposition}
\begin{proof}
We first show that $\Psi(\alpha,\B)$ always maps an inner edge to an inner edge.
To see this, let  $e=\{(g,h_1),(g,h_2)\}$ be an inner edge. 
Then $\Psi(\alpha,\B)$ maps edge $e$ to edge $\{(\alpha(g),\beta_g(h_1)),$$(\alpha(g),\beta_g(h_2))\}$,
which is an inner edge since $\beta_g$ is an automorphism of $H$.

Suppose now that $\Psi(\alpha,\B)$ is an automorphism. Since $\Psi(\alpha,\B)$ maps
inner edges to inner edges, it must  map connecting edges to connecting edges.
Let $e=\{(g,f(g_1)),(g_1,f(g))\}$ be a connecting edge.
Then $\Psi(\alpha,\B)(e)=\{\alpha(g),\beta_g(f(g_1)),(\alpha(g_1),\beta_{g_1}(f(g))\}$
is also a connecting edge.
Therefore $f(\alpha(g_1))=\beta_g(f(g_1))$.
Since $g_1$ can be any neighbour of $g$  in $G$, we have
$f \circ \alpha = \beta_g \circ f$ on $N(g)$.

Conversely, let $f \circ \alpha = \beta_g \circ f$ on $N(g)$ for every $g \in V(G)$.
Let $e=\{(g,f(g_1)),(g_1,f(g))\}$ be a connecting edge in $G \otimes_f H$.
Then  $\Psi(\alpha,\B)(e)=\{\alpha(g),\beta_g(f(g_1)),(\alpha(g_1),\beta_{g_1}(f(g))\}$.
Since $f(\alpha(g))=\beta_{g_1}(f(g))$ and $f(\alpha(g_1))=\beta_{g}(f(g_1))$,
$\Psi(\alpha,\B)(e)$ is a connecting edge. Therefore $\Psi(\alpha,\B)$ is an automorphism.
\end{proof}

\begin{proposition}
\label{prop:auto12a}
The mapping $\Psi(\alpha,\beta)$ is an automorphism if and only if
$f \circ \alpha = \beta \circ f$.
\end{proposition}
\begin{proof}
Let $f \circ \alpha = \beta \circ f$ on $N(G)$ for every $g \in V(G)$.
Since $G$ is connected it has no isolated points and so
$f \circ \alpha = \beta \circ f$ on $V(G)$.
The claim then follows from Proposition \ref{prop:auto12}.
\end{proof}

A few special cases now follow as simple corollaries.

\begin{corollary}
\label{cor:auto12}
Suppose $G=H$ and $f$ is an automorphism. 
Then  the mapping $\Psi(\alpha,\beta)$ is an automorphism if and only if
$\beta = f \circ \alpha \circ f^{-1}$.
\end{corollary}

\begin{corollary}
\label{cor:auto14}
Suppose $G=H$ and $f$ is the identity mapping.
Then the mapping $\Psi(\alpha,\beta)$ is an automorphism if and only if $\alpha=\beta$.
\end{corollary}

\begin{corollary}
\label{cor:auto13}
Suppose $V(G) \subseteq V(H)$, $f$ is the identity mapping on its domain
and $\beta|_{V(G)}=\alpha$. Then the mapping $\Psi(\alpha,\beta)$
is an automorphism.
\end{corollary}

\begin{remark}
If $f$ is injective and $G \ne H$, we can always relabel the vertices of $G,H$
such that $f$ is the identity on its domain.
\end{remark}

We now give some examples showing that $f$ need not be injective or surjective
and we can still have automorphisms of type $\Psi(\alpha,\B)$. 
Also, if $G=H$, the mapping $f$ need not be an automorphism.


\begin{example}
Let $G=K_3$ and $H=K_{3,3}$ with $V(G)=\{1,2,3\}$ and $V(H)=\{1,2,\dots,6\}$.
Let $f:V(G) \to V(H)$ map $1 \mapsto 1$, $2 \mapsto 3$, $3 \mapsto 5$.
Let $\alpha =(1 \ 2 \ 3)$, $\beta_1=(1 \ 3 \ 5)(2 \ 4\ 6)$,
$\beta_2=(1 \ 3 \ 5)(2 \ 6\ 4)$, $\beta_3=(1 \ 3 \ 5)$
and let $\B:V(G) \to \Aut(G)$, $B(g) = \beta_g$.
Then $f \circ \alpha = \beta_1 \circ f  = \beta_2 \circ f = \beta_3 \circ f$ and 
$$
\Psi(\alpha,\B) =(11 \ 23 \ 35)(12 \ 24 \ 32)(13 \ 25 \ 31)(14 \ 26 \ 34)(15 \ 21 \ 33)(16 \ 22 \ 36)
$$ is an automorphism of $G \otimes_f H$ that cyclically permutes the subgraphs $gH$, see Figure \ref{fig:ex1}.
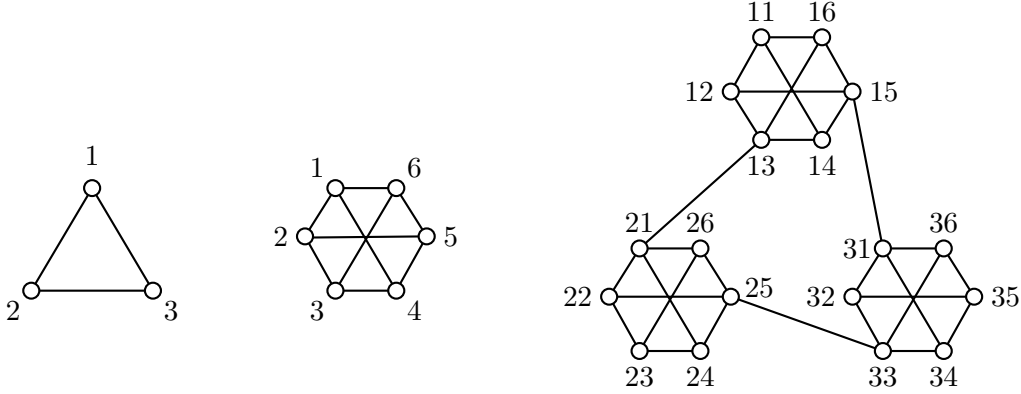
\begin{figure}[htb]  
\begin{center}
\begin{tikzpicture}[style=thick,scale=0.8,x=1.0cm,y=1.0cm]
	\def\vr{3pt/0.8}
	\begin{small}
	\draw (-5,-1)--(-3,-1)--(-4,0.7)--cycle;
	\draw [fill=white] (-4,0.7) circle (\vr);
		\draw (-4,0.9) node [anchor = south] {1};
	\draw [fill=white] (-5,-1) circle (\vr) node [anchor = north east] {2};
	\draw [fill=white] (-3,-1) circle (\vr) node [anchor = north west] {3};
	\draw (0,0.7)--(-0.5,-0.1)--(0,-1)--(1,-1)--(1.5,-0.1)--(1,0.7)--cycle;
	\draw (0,-1)--(1,0.7);
	\draw (0,0.7)--(1,-1);
	\draw (1.5,-0.1)--(-0.5,-0.12);
	\draw [fill=white] (0,-1) circle (\vr) node [anchor = north east] {3};
	\draw [fill=white] (1,-1) circle (\vr) node [anchor = north west] {4};
	\draw [fill=white] (1.5,-0.1) circle (\vr);
		\draw (1.6,-0.1) node [anchor = west] {5};
	\draw [fill=white] (1,0.7) circle (\vr) node [anchor = south west] {6};
	\draw [fill=white] (0,0.7) circle (\vr) node [anchor = south east] {1};
	\draw [fill=white] (-0.5,-0.1) circle (\vr);
		\draw (-0.6,-0.1) node [anchor = east] {2};
%
	\draw (7,1.5)--(5,-0.3);
	\draw (8.5,2.3)--(9,-0.3);
	\draw (6.5,-1.1)--(9,-2);
	\draw (7,1.5)--(8,1.5)--(8.5,2.3)--(8,3.2)--(7,3.2)--(6.5,2.3)--cycle;
	\draw (7,1.5)--(8,3.2);
	\draw (7,3.2)--(8,1.5);
	\draw (8.5,2.3)--(6.5,2.3);
	\draw [fill=white] (7,1.5) circle (\vr);
		\draw (7,1.4) node [anchor = north] {13};
	\draw [fill=white] (8,1.5) circle (\vr);
		\draw (8,1.4) node [anchor = north] {14};
	\draw [fill=white] (8.5,2.3) circle (\vr);
		\draw (8.6,2.3) node [anchor = west] {15};
	\draw [fill=white] (8,3.2) circle (\vr);
		\draw (8,3.3) node [anchor = south] {16};
	\draw [fill=white] (7,3.2) circle (\vr);
		\draw (7,3.3) node [anchor = south] {11};
	\draw [fill=white] (6.5,2.3) circle (\vr);
		\draw (6.4,2.3) node [anchor = east] {12};
	\draw (5,-0.3)--(4.5,-1.1)--(5,-2)--(6,-2)--(6.5,-1.1)--(6,-0.3)--cycle;
	\draw (5,-0.3)--(6,-2);
	\draw (5,-2)--(6,-0.3);
	\draw (6.5,-1.1)--(4.5,-1.1);
	\draw [fill=white] (5,-2) circle (\vr);
		\draw (5,-2.1) node [anchor = north] {23};
	\draw [fill=white] (6,-2) circle (\vr);
		\draw (6,-2.1) node [anchor = north] {24};
	\draw [fill=white] (6.5,-1.1) circle (\vr);
		\draw (6.55,-0.95) node [anchor = west] {25};
	\draw [fill=white] (6,-0.3) circle (\vr);
		\draw (6.,-0.2) node [anchor = south] {26};
	\draw [fill=white] (5,-0.3) circle (\vr);
		\draw (5,-0.2) node [anchor = south] {21};
	\draw [fill=white] (4.5,-1.1) circle (\vr);
		\draw (4.4,-1.1) node [anchor = east] {22};
	\draw (9,-0.3)--(8.5,-1.1)--(9,-2)--(10,-2)--(10.5,-1.1)--(10,-0.3)--cycle;
	\draw (9,-0.3)--(10,-2);
	\draw (9,-2)--(10,-0.3);
	\draw (10.5,-1.1)--(8.5,-1.1);
	\draw [fill=white] (8.5,-1.1) circle (\vr);
		\draw (8.4,-1.1) node [anchor = east] {32};
	\draw [fill=white] (9,-2) circle (\vr);
		\draw (9,-2.1) node [anchor = north] {33};
	\draw [fill=white] (10,-2) circle (\vr);
		\draw (10,-2.1) node [anchor = north] {34};
	\draw [fill=white] (10.5,-1.1) circle (\vr);
		\draw (10.6,-1.1) node [anchor = west] {35};
	\draw [fill=white] (10,-0.3) circle (\vr);
		\draw (10,-0.2) node [anchor = south] {36};
	\draw [fill=white] (9,-0.3) circle (\vr);
		\draw (9,-0.3) node [anchor = east] {31};
	\end{small}
\end{tikzpicture}
\caption{Graphs $K_3$, $K_{3,3}$ and their Sierpi\'nski product with respect to 
$f:V(K_3) \to V(K_{3,3})$, $f: 1 \mapsto 1, 2 \mapsto 3, 3 \mapsto 5$.
}
 \label{fig:ex1}
\end{center}
\end{figure}
\end{example}

\begin{example}
Let $G = H = K_{1,3}$ with edge set $\{\{1,2\},\{2,3\},\{2,4\}\}$, and let $f:V(G) \to V(G)$ be
defined as  $f=(1 \ 2 \ 3 \ 4)$.
Note that $f$ is a bijection that is not an automorphism of $G$.
If $\alpha =(3 \ 4)$ and $\beta= f \circ \alpha \circ f^{-1} = (1 \ 4)$, then 
$f \circ \alpha = \beta \circ f$ and 
$$
\Psi(\alpha,\beta)=(11 \ 14)(21 \ 24)(31 \ 44)(32 \ 42)(33 \ 43)(34 \ 41)
$$ 
is an automorphism of $G \otimes_f G$, 
that swaps copies $3G$ and $4G$, see Figure \ref{fig:ex2}.

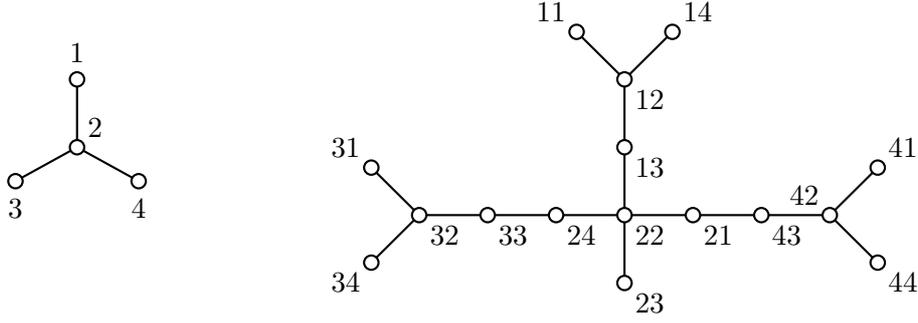
\begin{figure}[htb!]
\begin{center}
\begin{tikzpicture}[style=thick,scale=0.9,x=1.0cm,y=1.0cm]
	\def\vr{3pt}
	\begin{small}
	\draw (-4,3)--(-4,2);
	\draw (-4,2)--(-4.9,1.5);
	\draw (-4,2)-- (-3.1,1.5);
	\draw [fill=white] (-4,3) circle (\vr);
		\draw (-4,3.1) node [anchor = south] {1};
	\draw [fill=white] (-4,2) circle (\vr);
		\draw (-4,2) node [anchor = south west] {2};
	\draw [fill=white] (-4.9,1.5) circle (\vr);
		\draw (-4.9,1.4) node [anchor = north] {3};
	\draw [fill=white] (-3.1,1.5) circle (\vr);
		\draw (-3.1,1.4) node [anchor = north] {4};
%
%
	\draw (1,1)--(2,1)--(3,1)--(4,1)--(4,2)--(4,3);
	\draw (4,0)--(4,1)--(5,1)--(6,1)--(7,1);
	\draw (0.3,0.3)--(1,1)--(0.3,1.7);
	\draw (3.3,3.7)--(4,3)--(4.7,3.7);
	\draw (7.7,1.7)--(7,1)--(7.7,0.3);
	\draw [fill=white] (4,1) circle (\vr) node [anchor = north west] {22};
	\draw [fill=white] (4,2) circle (\vr) node [anchor = north west] {13};
	\draw [fill=white] (4,3) circle (\vr) node [anchor = north west] {12};
	\draw [fill=white] (3,1) circle (\vr) node [anchor = north west] {24};
	\draw [fill=white] (5,1) circle (\vr) node [anchor = north west] {21};
	\draw [fill=white] (4,0) circle (\vr) node [anchor = north west] {23};
	\draw [fill=white] (2,1) circle (\vr) node [anchor = north west] {33};
	\draw [fill=white] (1,1) circle (\vr) node [anchor = north west] {32};
	\draw [fill=white] (6,1) circle (\vr) node [anchor = north west] {43};
	\draw [fill=white] (7,1) circle (\vr) node [anchor = south east] {42};
	\draw [fill=white] (0.3,1.7) circle (\vr) node [anchor = south east] {31};
	\draw [fill=white] (0.3,0.3) circle (\vr) node [anchor = north east] {34};
	\draw [fill=white] (7.7,0.3) circle (\vr) node [anchor = north west] {44};
	\draw [fill=white] (7.7,1.7) circle (\vr) node [anchor = south west] {41};
	\draw [fill=white] (4.7,3.7) circle (\vr) node [anchor = south west] {14};
	\draw [fill=white] (3.3,3.7) circle (\vr) node [anchor = south east] {11};
	\end{small}
\end{tikzpicture}
\caption{Graph $G=K_{1,3}$ and the Sierpi\'nski product $G \otimes_f G$ with respect to 
$f=(1 \ 2 \ 3 \ 4)$.} 
\label{fig:ex2}
\end{center}
\end{figure}
\end{example}

\begin{example}
Let $G=C_4$  with $V(G)=\{1,2,3,4\}$ and let $H$ be a star, with edge set $\{\{1,2\},\{2,3\},\{2,4\}\}$.
Let $f:V(G) \to V(H)$ map $1 \mapsto 2$, $2 \mapsto 2$, $3 \mapsto 4$ and $4 \mapsto 3$. 
Note that the mapping $f$ is neither injective nor surjective.
If $\alpha =(1 \ 2)(3 \ 4)$ and $\beta=( 3\ 4)$, then $f \circ \alpha = \beta \circ f$ and 
$$
\Psi(\alpha,\beta)=(11 \ 21)(12 \ 22) (13 \ 24) (14 \ 23)(31 \ 41)(32 \ 42)(33 \ 44)(34 \ 43)
$$ is a reflection automorphism of $G \otimes_f H$, swapping 
copies $1H,2H$ and $3H,4H$, see Figure \ref{fig:ex3}.

\begin{figure}[htb!]
\begin{center}
\begin{tikzpicture}[style=thick,scale=0.6,x=1.0cm,y=1.0cm]
	\def\vr{3pt/0.65}
	\begin{small}
	\draw (10.5,9.5)--(10.5,7.5);
	\draw (10.5,7.5)--(8.8,6.5);
	\draw (10.5,7.5)--(12.2,6.5);
	\draw [fill=white] (10.5,9.5) circle (\vr);
		\draw (10.5,9.7) node[anchor=south] {1};
	\draw [fill=white] (10.5,7.5) circle (\vr) node[anchor=south west] {2};
	\draw [fill=white] (8.8,6.5) circle (\vr) node[anchor=north east] {3};
	\draw [fill=white] (12.2,6.5) circle (\vr) node[anchor=north west] {4};
	\draw (9,3)--(9,0);
	\draw (9,0)--(12,0);
	\draw (12,0)--(12,3);
	\draw (12,3)--(9,3);
	\draw [fill=white] (9,3) circle (\vr) node[anchor=south east] {4};
	\draw [fill=white] (9,0) circle (\vr) node[anchor=north east] {1};
	\draw [fill=white] (12,0) circle (\vr) node[anchor=north west] {2};
	\draw [fill=white] (12,3) circle (\vr) node[anchor=south west] {3};
%
%
	\draw (25,7)--(23,9)--(20.17,9)--(18.17,7)--(18.17,4.17)
					--(20.17,2.17)--(23,2.17)--(25,4.17)--cycle;
	\draw (20.48,-0.15)--(20.17,2.17)--(18.31,0.75);
	\draw (15.85,6.69)--(18.17,7)--(16.75,8.86);
	\draw (26.42,8.86)--(25,7)--(27.32,6.69);
	\draw (24.86,0.75)--(23,2.17)--(22.69,-0.15);
	\draw [fill=white] (25,7) circle (\vr) node [anchor = north east] {32};
	\draw [fill=white] (23,9) circle (\vr);
	\draw (23.2,8.8) node [anchor = north east] {33};
	\draw [fill=white] (20.17,9) circle (\vr);
	\draw (20,8.8) node [anchor = north west] {44};
	\draw [fill=white] (18.17,7) circle (\vr) node [anchor = north west] {42};
	\draw [fill=white] (18.17,4.17) circle (\vr) node [anchor = south west] {13};
	\draw [fill=white] (20.17,2.17) circle (\vr);
	\draw (20,2.37) node [anchor = south west] {12};
	\draw [fill=white] (23,2.17) circle (\vr);
	\draw (23.2,2.37) node [anchor = south east] {22};
	\draw [fill=white] (25,4.17) circle (\vr) node [anchor = south east] {24};
	\draw [fill=white] (20.48,-0.15) circle (\vr);
	\draw [fill=white] (18.31,0.75) circle (\vr);
	\draw [fill=white] (24.86,0.75) circle (\vr);
	\draw [fill=white] (26.42,8.86) circle (\vr);
	\draw [fill=white] (15.85,6.69) circle (\vr);
	\draw [fill=white] (16.75,8.86) circle (\vr);
	\draw [fill=white] (27.32,6.69) circle (\vr);
	\draw [fill=white] (22.69,-0.15) circle (\vr);
	\draw (17.6,0.6) node[anchor=north west] {11};
	\draw (19.8,-0.4) node[anchor=north west] {14};
	\draw (25,0.6) node[anchor=north west] {21};
	\draw (22.5,-0.4) node[anchor=north west] {23};
	\draw (27.2,6.5) node[anchor=north west] {31};
	\draw (26.6,9.4) node[anchor=north west] {34};
	\draw (15.1,6.5) node[anchor=north west] {41};
	\draw (15.5,9.4) node[anchor=north west] {43};
	\end{small}
\end{tikzpicture}
\caption{Graphs $C_4$, $K_{1,3}$ and their Sierpi\'nski product  with respect to 
$f:V(C_4) \to V(K_{1,3})$, $f: 1 \mapsto 2, 2 \mapsto 2, 3 \mapsto 4, 4 \mapsto 3$.} 
 \label{fig:ex3}
\end{center}
\end{figure}
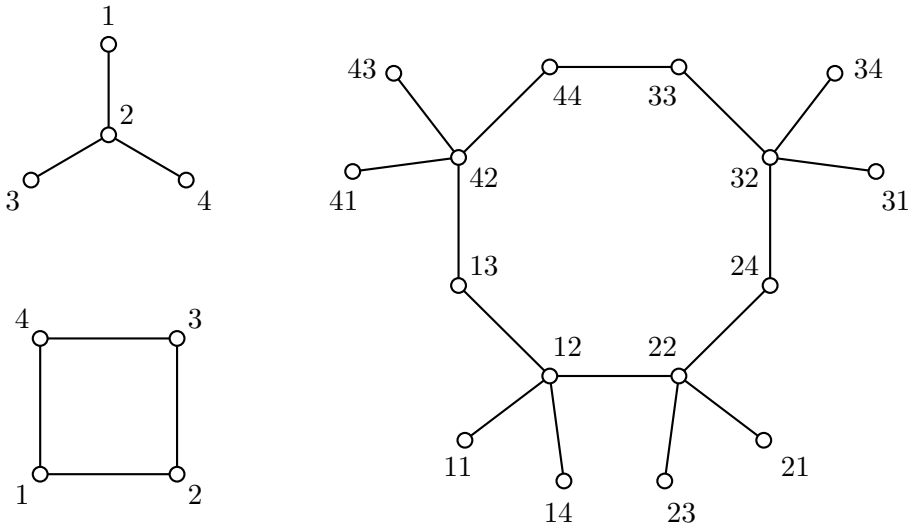
\end{example}


Now let us introduce the second family of automorphisms.
Let   $g \in V(G)$ and $\beta \in \Aut(H)$.
Define a mapping $\Phi(g,\beta): V(G \otimes_f H) \to V(G \otimes_f H)$ given by
$$
\Phi(g,\beta): (g_1,h_1) \mapsto  \begin{cases}
                   (g_1,h_1)  & \mbox{if} \ \  g_1 \ne g,\\
                   (g_1,\beta(h_1))  & \mbox{if} \ \ g_1=g.
                 \end{cases}  
$$

\begin{proposition}
\label{prop:auto21}
The mapping $\Phi(g,\beta)$ is an automorphism of $G \otimes_f H$
if and only if $\beta$ is in the stabilizer of $f(N(g))$.
Moreover, in this case $\Phi(g,\beta)$ respects the fundamental edge partition.
\end{proposition}
\begin{proof}
The mapping $\Phi(g,\beta)$ is obviously a bijection since it
fixes all the vertices of $G \otimes_f H$ not in $gH$ and it permutes the vertices in $gH$.
It also fixes inner edges  and  connecting edges that do 
not have any endvertex in $gH$ and it permutes inner edges in $gH$.

Take a vertex $g' \in N(G)$.
Then $\{(g,f(g')),(g',f(g))\}$ is a connecting edge.
The mapping  $\Phi(g,\beta)$ maps $\{(g,f(g')),(g',f(g))\}$ to
the set $\{(g,\beta(f(g')),(g',f(g))\}$, which is an edge if and only if
$\beta(f(g')=f(g')$.
So $\Phi(g,\beta)$ is an automorphism if and only if 
$\beta$ is in the stabilizer of $f(g')$ for every $g' \in N(G)$.
\end{proof}

\begin{remark}   \label{rem:specialcase}
Note that by Theorem \ref{thm:autoform}, a mapping $\Phi(g,\beta)$
is the same as $\Psi(\alpha, \B)$ for some $\alpha \in \Aut(G)$ and
$\B: V(G) \to \Aut(H)$. Indeed, it is easy to verify that for
$\alpha = {\rm id}$  and 
$\B$ defined by the rules $\B: g_1 \to  {\rm id}$ if $g_1 \ne g$ and
$\B: g \to  \beta$, we have
$\Phi(g,\beta)=\Psi(\alpha, \B)$.
\end{remark}

Given a group $X$ acting on set $Y$, we denote by $X_Y$ the stabilizer
of $Y$, i.e., the subgroup of $X$ that fixes every element of $Y$.
For $g \in G$ denote by $\hat B_g(G,H,f)$ the group generated by 
$\{\Phi(g,\beta_g)| \, \beta_{g}\in \Aut(H)_{f(N(g))}\}$.
Denote by $\hat B(G,H,f)$ the group generated by 
$\{\hat B_g(G,H,f)| \, \ g \in V(G)\}$.

\begin{proposition}
\label{prop:auto22}
Let $g,g^\prime$ be distinct vertices of $G$ and let 
$\beta_g \in   \Aut(H)_{f(N(g))}$,  $\beta_{g^\prime}\in \Aut(H)_{f(N(g^\prime))}$.
Then  $\Phi(g,\beta_g)$ and  $\Phi(g',\beta_{g'})$ commute.
\end{proposition}
\begin{proof}
Mappings $\Phi(g,\beta_g)$ and  $\Phi(g',\beta_{g'})$ commute
since as permutations they have disjoint supports.
\end{proof}

\begin{theorem}
Group $\hat B(G,H,f)$ is a subgroup of  group $\tilde A(G,H,f)$ and
is  a direct product 
\begin{equation}    \label{eq_Bprod}
\hat B(G,H,f)=\prod_{g \in V(G)} \hat B_g(G,H,f).
\end{equation}
Moreover, the group $\hat B(G,H,f)$ is isomorphic to the group 
$\prod_{g \in V(G)} \Aut(H)_{f(N(g))}.$
\end{theorem}
\begin{proof}
Group $\hat B(G,H,f)$ is a subgroup of $\tilde A(G,H,f)$ by
the definition and Propositon \ref{prop:auto21}.
Since the groups $\hat B_g(G,H,f)$, $g \in V(G)$, have pairwise 
only the identity in common, they generate $\hat B(G,H,f)$, 
and the elements of two distinct groups commute,  equation \eqref{eq_Bprod} holds.
The last claim is true since for every $g \in G$ the groups
$\hat B_g(G,H,f)$ and $\Aut(H)_{f(N(g))}$ are isomorphic in the obvious way.
\end{proof}


\subsection{When do all the automorphisms respect the fundamental edge partition}

Given connected graphs $G,H$ and a mapping $f:V(G) \to V(H)$,
in general there can exist automorphisms of $G \otimes_f H$ that do not respect 
the fundamental edge partition. Figure \ref{fig:counterexample}  shows such an example.
There $G=C_4$, $H=2K_3+e$ and $f: V(G) \to V(H)$ is the identity function on its domain.
One can easily observe that   cyclic rotation of $G \otimes_f H$  maps  inner edge 
$\{16,15\}$ to  connecting edge $\{14,41\}$. 

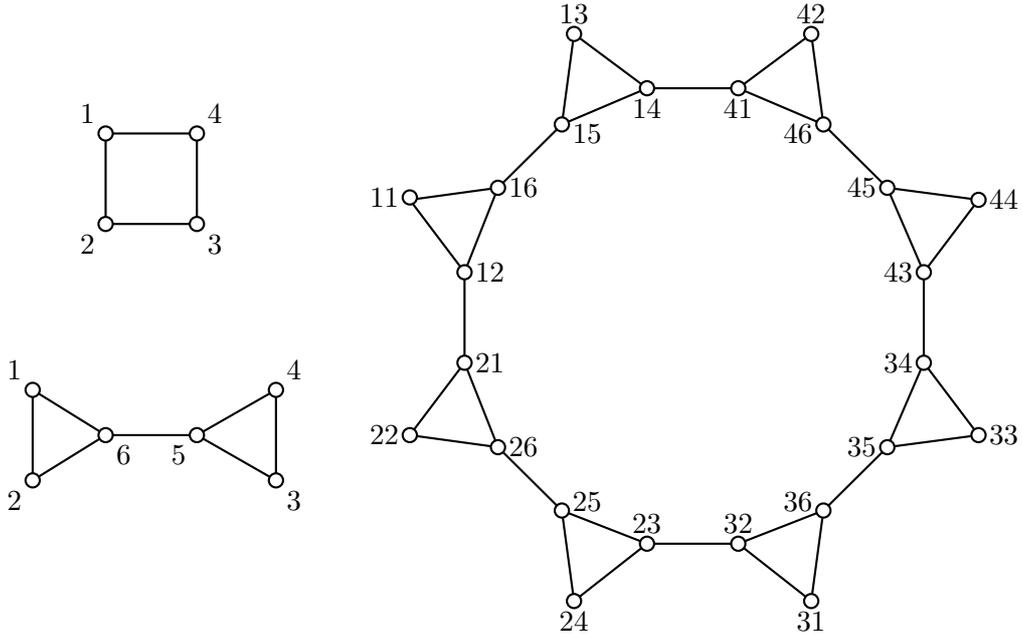
\begin{figure}
\centering
\begin{tikzpicture}[scale=0.4,style=thick]
	\def\vr{3pt/0.45}
	\begin{small}
	\draw (7.8,1.9)--(10.8,1.9)--(13.6,3)--(15.7,5.1)--(16.9,7.9)--(16.9,10.9)
			--(15.7,13.7)--(13.6,15.8)--(10.8,17)--(7.8,17)--(5,15.8)--(2.9,13.7)
			--(1.8,10.9)--(1.8,7.9)--(2.9,5.1)--(5,3)--cycle;
	\draw (1.8,10.9)--(0,13.3)--(2.9,13.7);
	\draw (5,15.8)--(5.4,18.8)--(7.8,17);
	\draw [fill=white] (7.8,17) circle (\vr) node[anchor=north] {14};
	\draw [fill=white] (5,15.8) circle (\vr);
		\draw (5,15.5) node[anchor=west] {15};
	\draw [fill=white] (2.9,13.7) circle (\vr) node[anchor=west] {16};
	\draw [fill=white] (1.8,10.9) circle (\vr) node[anchor=west] {12};
	\draw [fill=white] (0,13.38) circle (\vr) node[anchor=east] {11};
	\draw [fill=white] (5.4,18.8) circle (\vr) node[anchor=south] {13};
	\draw (7.8,1.9)--(5.4,0)--(5,3); 
	\draw (2.9,5.1)--(0,5.5)--(1.8,7.9);
	\draw [fill=white] (1.8,7.9) circle (\vr) node[anchor=west] {21};
	\draw [fill=white] (2.9,5.1) circle (\vr) node[anchor=west] {26};
	\draw [fill=white] (5,3) circle (\vr);
		\draw (5,3.3) node[anchor=west] {25};
	\draw [fill=white] (7.8,1.9) circle (\vr) node[anchor=south] {23};
	\draw [fill=white] (5.4,0) circle (\vr) node[anchor=north] {24};
	\draw [fill=white] (0,5.5) circle (\vr) node[anchor=east] {22};
	\draw (16.9,7.9)--(18.7,5.5)--(15.7,5.1);
	\draw (13.6,3)--(13.2,0)--(10.8,1.9);
	\draw [fill=white] (10.8,1.9) circle (\vr) node[anchor=south] {32};
	\draw [fill=white] (13.6,3) circle (\vr);
		\draw (13.6,3.3) node[anchor=east] {36};
	\draw [fill=white] (15.7,5.1) circle (\vr) node[anchor=east] {35};
	\draw [fill=white] (16.9,7.9) circle (\vr) node[anchor=east] {34};
	\draw [fill=white] (18.7,5.5) circle (\vr) node[anchor=west] {33};
	\draw [fill=white] (13.2,0) circle (\vr) node[anchor=north] {31};
	\draw (10.8,17)--(13.2,18.8)--(13.6,15.8);
	\draw (15.7,13.7)--(18.7,13.3)--(16.9,10.9);
	\draw [fill=white] (16.9,10.9) circle (\vr) node[anchor=east] {43};
	\draw [fill=white] (15.7,13.7) circle (\vr) node[anchor=east] {45};
	\draw [fill=white] (13.6,15.8) circle (\vr);
		\draw (13.6,15.5) node[anchor=east] {46};
	\draw [fill=white] (10.8,17) circle (\vr) node[anchor=north] {41};
	\draw [fill=white] (13.2,18.8) circle (\vr) node[anchor=south] {42};
	\draw [fill=white] (18.7,13.3) circle (\vr) node[anchor=west] {44};
%
%
	\draw (-7,15.5)--(-10,15.5)--(-10,12.5)--(-7,12.5)--cycle;
	\draw [fill=white] (-10,15.5) circle (\vr) node[anchor=south east] {1};
	\draw [fill=white] (-10,12.5) circle (\vr) node[anchor=north east] {2};
	\draw [fill=white] (-7,12.5) circle (\vr) node[anchor=north west] {3};
	\draw [fill=white] (-7,15.5) circle (\vr) node[anchor=south west] {4};
	\draw (-10,5.5)--(-7,5.5);
	\draw (-7,5.5)--(-4.4,7)--(-4.4,4)--cycle;
	\draw (-10,5.5)--(-12.4,7)--(-12.4,4)--cycle;
	\draw [fill=white] (-10,5.5) circle (\vr) node[anchor=north west] {6};
	\draw [fill=white] (-7,5.5) circle (\vr) node[anchor=north east] {5};
	\draw [fill=white] (-4.4,7) circle (\vr) node[anchor=south west] {4};
	\draw [fill=white] (-4.4,4) circle (\vr) node[anchor=north west] {3};
	\draw [fill=white] (-12.4,7) circle (\vr) node[anchor=south east] {1};
	\draw [fill=white] (-12.4,4) circle (\vr) node[anchor=north east] {2};
	\end{small}
\end{tikzpicture}
\caption{Graphs $C_4$, $2K_3+e$ and their Sierpi\'nski product with respect to $f = {\rm id}$.}
\label{fig:counterexample}
\end{figure}

Note that in the example above, graph $H$ is not 2-connected.
When  graphs $G,H$ are both $2$-connected,
we have so far not been able to find an automorphism of
$G \otimes_f H$ that does not respect the fundamental edge partition.
Therefore we propose the following Conjecture.

\begin{conjecture}
Let $G,H$ be $2$-connected  graphs and let $f: V(G)  \to V(H)$ be any mapping.
Then
$$\tilde A(G,H,f)= \Aut(G \otimes_f H).$$
\end{conjecture}

In this section we prove this conjecture for  two special cases.
In the first case $G=H$ and $G$ is a regular triangle-free graph.
In the second case  every edge of $H$ is contained in a short cycle.
Note that in these two cases the assumption that $G,H$ are
$2$-connected is not needed.

\begin{proposition}
	\label{lem:fep_trianglefree}
Let $G$ be a connected regular triangle-free graph and 
let $f: V(G)  \to V(G)$ be an automorphism of $G$.
Then every automorphism of $G \otimes_f G$ respects the fundamental edge partition.
\end{proposition}
\begin{proof}
Let $k$ denote the valency of $G$. Then the endvertices of
every connecting edge in $G \otimes G_f$ have valency $k+1$
by Lemma~\ref{lem:atmostone} (ii).
An endvertex of an inner edge may have valency $k$ or $k+1$.
Clearly, if at least one endvertex of an inner edge has  valency $k$,
this edge cannot be mapped to a connecting edge by any automorphism.

Suppose now that both endvertices of an inner edge $\{(g,g_1),(g,g_2)\}$ have
degree $k+1$. This is only possible if $(g,g_1)$ and $(g,g_2)$
are endvertices of some connecting edges, say
$\{(g,g_1),(g_1',f(g))\}$ and  $\{(g,g_2),(g_2',f(g))\}$ where
$g_1=f(g_1')$ and $g_2=f(g_2')$.
But then $g_1'$ and $g_2'$ are adjacent to $g$ in $G$.
Since $g_1$ and $g_2$ are adjacent in $G$ and $f$ is an automorphism,
also $g_1'$ and $g_2'$ are adjacent.
But then $g,g_1'g_2'$ form a triangle in $G$, a contradiction.
Therefore no inner edge can be mapped to a connecting edge,
so  every automorphism of $G \otimes_f G$ respects the fundamental edge partition.
\end{proof}

\begin{lemma}
	\label{lem:shortcycles}
Let $G$ and $H$ be  graphs and let $f: V(G)  \to V(H)$ be any mapping.
Let $\{g,g'\}$ be an edge of $G$.
\begin{itemize}
\item[(i)] If $\{g,g'\}$ is not contained in any cycle of $G$, then the edge
             $\{(g,f(g'),(g',f(g))\}$ is not contained in any cycle of  $G \otimes_f H$.
\item[(ii)] Let $c$ be the length of the shortest cycle that contains $\{g,g'\}$.
              Then the  shortest cycle that contains the edge $\{(g,f(g'),(g',f(g))\}$ 
              in $G \otimes_f H$ has length at least $c$.
 \item[(iii)] Suppose  that $f$ is  locally injective and let $c$ be the length 
              of the shortest cycle that contains $\{g,g'\}$.
              Then the  shortest cycle that contains the edge $\{(g,f(g'),(g',f(g))\}$ 
              in $G \otimes_f H$ has length at least $2c$.
\end{itemize}
\end{lemma}

\begin{proof}
Let $C$ be a cycle  in $G \otimes_f H$ that contains $\{(g,f(g'),(g',f(g))\}$.
Suppose that 
$\{(g,f(g'),(g',f(g))\},\{(g',f(g_1),(g_1,f(g'))\}, 
\dots,$  $\{(g_k,f(g),(g,f(g_k))\}$
are the connecting edges in $C$   in that order. Then $g g' g_1 g_2 \dots g_k g$ is a cycle of length $k$  in $G$ that contains the edge  $\{g,g'\}$, so $k \ge c$. Furthermore,
if $\{g,g'\}$ is not contained in any cycle of $G$, then the edge
$\{(g,f(g'),(g',f(g))\}$ can not be  contained in any cycle of  $G \otimes_f H$. 
Recall that if $f$ is locally injective,   any vertex of $G \otimes_f H$ is an end-vertex of at most
one connecting edge by Lemma \ref{lem:atmostone}.  
Therefore in this case the shortest cycle that contains $\{(g,f(g'),(g',f(g))\}$
has length at least $2c$.
\end{proof}

\begin{proposition}
	\label{lem:fep}
Let $G$ and $H$ be  connected graphs, let $f: V(G)  \to V(H)$ be any mapping
and let the girth of $G$ be equal to $c$. In any of the following cases
every automorphism of $G \otimes_f H$  respects the fundamental edge partition:
\begin{itemize}
\item[(i)] $G$ is a tree and $H$ is a bridgeless graph;
\item[(ii)] every edge of $H$ is contained in a cycle of length at most $c-1$;
 \item[(iii)] mapping $f$ is  locally injective and every edge of $H$ 
              is contained in a cycle of length at most $2c-1$.
\end{itemize}
\end{proposition}
\begin{proof}
By Lemma \ref{lem:shortcycles}, the  shortest cycle that contains a connecting
edge has length at least $c$ in case (ii),  length at least $2c$ in case (iii) and is not contained
in any cycle in case (i). Since every inner edge is contained in a cycle, in a cycle
of length at most $c-1$, in a cycle of length at most $2c-1$ in cases (i), (ii), (iii), respectively,
a connecting edge cannot be mapped to an inner edge by any automorphism.
\end{proof}

Using Propositions 	\ref{lem:fep_trianglefree} and \ref{lem:fep},
we see that in some cases the group of automorphisms that respect
the fundamental edge partition is in fact the whole automorphism group
of  $G \otimes_f H$.

\begin{theorem}
	\label{thm:fep_trianglefree}
Let $G$ be a connected regular triangle-free graph and 
let $f: V(G)  \to V(G)$ be an automorphism of $G$.
Then
$$\tilde A(G,G,f)= \Aut(G \otimes_f G).$$
\end{theorem}

\begin{theorem}
\label{thm:auto3}
Let $G$ and $H$ be  connected graphs, let $f: V(G)  \to V(H)$ be any mapping
and let the girth of $G$ be equal to $c$. In any of the following cases
\begin{itemize}
\item[(i)] $G$ is a tree and $H$ is a bridgeless graph:
\item[(ii)] every edge of $H$ is contained in a cycle of length at most $c-1$;
 \item[(iii)] mapping $f$ is locally  injective and every edge of $H$ 
              is contained in a cycle of length at most $2c-1$;
\end{itemize}
the group $\tilde A(G,H,f)$ is equal to 
$\Aut(G \otimes_f H).$

\end{theorem}

\subsection{Group of automorphisms of $G\otimes_f G$}

We now consider the group of automorphisms that respect the fundamental edge
partition  in the special case when $G=H$ and $f:V(G)\to V(G)$ is an automorphism.
Since in this case $G\otimes_f G$ is isomorphic to $G\otimes G$ we could 
restrict ourselves to the case where $f$ is the identity.
Note that in that case the structure of the automorphism group was sketched in the paper \cite{Kovse}, 
but the proofs were never published.

Recall that by Corollary \ref{cor:auto12}, every automorphism $\alpha$ of $G$ 
has a lift,  $\Psi(\alpha, f \circ \alpha \circ f^{-1})$.
We call this automorphism the \emph{diagonal automorphism}
of $G \otimes_f G$ corresponding to $\alpha$ and denote it by $\bar{\alpha}$.
Denote by $\bar A(G,f)$ the set of all diagonal automorphisms.
The following proposition is straightforward to prove.

\begin{proposition}
The set $\bar A(G,f)$ is a subgroup of $\tilde A(G,G,f)$, isomorphic to $\Aut(G)$.
\end{proposition}

To determine the structure of the group $\tilde A(G,G,f)$, we first show that every element
of $\tilde A(G,G,f)$ can be written as a product of an element from $\hat B(G,G,f)$
and an element of $\bar A(G,f)$. Furthermore, we show that $\hat B(G,G,f)$ is normal in $\tilde A(G,G,f)$. 

\begin{theorem}
\label{thm:auto1}
Let $G$ be a connected graph and let $f: V(G)  \to V(G)$ be an automorphism.
Let $\tilde \gamma$ be an automorphism of $G \otimes_f G$.
Then there exist $\alpha \in \Aut(G)$ and $\beta_g \in \Aut(G)_{f(N(g))}$ for  every $g \in V(G)$
such that 
$\tilde \gamma = \bar \alpha \,  \left(\prod_{g \in V(G)}\, \Phi(g,\beta_g) \right).$
\end{theorem}
\begin{proof}
Let $\alpha$ be the projection of $\tilde \gamma$ to $\Aut(G)$.
Then $\bar \alpha=\Psi(\alpha, f \circ \alpha \circ f^{-1})$  permutes
the copies $gG$ in the right way. Observe that $\bar \alpha$ already agrees 
with $\tilde \gamma$ on the endvertices of all the connecting edges.
To obtain $\tilde \gamma$ from $\bar \alpha$, we only need to adjust the
action of $\bar \alpha$ on the vertices that are not endvertices of connecting edges.
We can do this on every copy $gG$ separately, by acting with $\Phi(g,\beta_g)$,
where  $\beta_g \in \Aut(G)$ is induced by $\bar \alpha^{-1} \tilde \gamma  $.
Also $\beta_g \in \Aut(G)_{f(N(g))}$   since the vertices $f(N(g))$ have the right image already and are fixed.
\end{proof}

\begin{theorem}  \label{thm:Bnormal}
Let $f:V(G) \to V(G)$ be an automorphism.
Then $\hat B(G,G,f)$ is a normal subgroup of  group $\tilde A(G,G,f)$ 
\end{theorem}
\begin{proof}
Observe that the mapping $\lambda: \tilde A(G,G,f) \to \tilde A(G,G,f)$
defined by $\lambda:\Psi(\alpha, \B) \to \Psi(\alpha, f \circ \alpha \circ f^{-1})$
is a homomorphism of groups with $\hat B(G,G,f)$ being its kernel. 
Therefore $\hat B(G,G,f)$ is a normal subgroup of $\tilde A(G,G,f)$.
\end{proof}

\begin{theorem}
\label{thm:auto2}
Let $G$ be a connected graph and let $f: V(G)  \to V(G)$ be an automorphism.
Then  the group $\tilde A(G,G,f)$ is a semidirect product,
$$\tilde A(G,G,f) = \bar A(G,f) \ltimes \hat B(G,G,f).$$
\end{theorem}
\begin{proof}
Group $\hat B(G,G,f)$ is a normal subgroup of $\tilde A(G,G,f)$ by Theorem \ref{thm:Bnormal}.
By Theorem \ref{thm:auto1}, every element of $\tilde A(G,G,f)$
can be written as a product of a diagonal automorphism and
an element from $\hat B(G,G,f)$.
Moreover, only identity is in both  $\bar A(G,f)$  and $\hat B(G,G,f)$. This proves that
$\tilde A(G,G,f)$ is a semidirect product of $\bar A(G,f)$ and $\hat B(G,G,f)$.
\end{proof}

Note that in general not every automorphism of $G$ has a lift.
Diagonal mappings are well defined only if $f$ is a bijection.
Define $\bar A(G,H,f)$ to be the set of all diagonal mappings that
are also automorphisms. This is a subgroup of $\tilde A(G,H,f)$.
We believe that the following conjecture holds.

\begin{conjecture}
Let $G$ and $H$ be  connected  graphs on the same vertex set
 and let $f: V(G)  \to V(H)$ be a bijection.
Then the group $\tilde A(G,H,f)$ is a semidirect product,
$$\tilde A(G,H,f) = \bar A(G,H,f) \ltimes \hat B(G,H,f).$$
\end{conjecture}




\section{Sierpi\'nski product with multiple factors}
\label{section:morefactors}

When extending the Sierpi\'nski product to more than two factors we first need to specify how 
the graph $G$ embeds into the product $G\otimes_f H$ in order to be able to multiply it with 
the next graph. This is exactly the role of the function $\varphi$ from Definition~\ref{def:sierpinski-product}.
Let $G_1$, $G_2$ and $G_3$ be graphs and $f: V(G_2) \rightarrow V(G_1)$, 
$f': V(G_3) \rightarrow V(G_2)$ be functions. 
Then the Sierpi\'nski product of these graphs is constructed so that we first build 
$(K,\varphi) = G_2\otimes_f G_1$ and then form the product 
$(K',\varphi') = G_3 \otimes_{\varphi\circ f'} K$.
Note that with given functions $f$ and $f'$ we cannot 
form this product in any other way, therefore Sierpi\'nski product is not associative.
We will denote such Sierpi\'nski product by $G_3 \otimes_{f'} G_2 \otimes_f G_1$. 


In Figure~\ref{fig:more_factors} it is shown how the product 
$C_3\otimes_{f'} C_4\otimes_{f} C_3$ is formed in two steps
(with  $f: V(C_4) \to V(C_3)$,  $f: i \mapsto  i\ (\bmod\ 3)$ and $f':V(C_3) \to V(C_4)$ 
being the identity function on its domain). 
\begin{figure}[htb!]
\begin{center}
	\begin{tikzpicture}[scale=0.50,style=thick]
	\def\vr{3pt/0.6}
	\begin{large}
		\draw (-3.9,4.3) node {$\otimes_{f}$};
		\draw (0.4,4.3) node {$=$};
	\end{large}
	\draw (-7,5.2)--(-7,3.5);
	\draw (-7,3.5)--(-5.3,3.5);
	\draw (-5.3,3.5)--(-5.3,5.2);
	\draw (-5.3,5.2)--(-7,5.2);
	\draw [fill=white] (-7,5.2) circle (\vr) node [anchor = south] {0};
	\draw [fill=white] (-7,3.5) circle (\vr) node [anchor = north] {1};
	\draw [fill=white] (-5.3,3.5) circle (\vr) node [anchor = north] {2};
	\draw [fill=white] (-5.3,5.2) circle (\vr) node [anchor = south] {3};
	\draw (-6.1,1.5) node {$G_2=C_4$};
	\draw (-2,5.23)--(-3,3.5);
	\draw (-1,3.5)--(-3,3.5);
	\draw (-2,5.23)--(-1,3.5);
	\draw [fill=white] (-2,5.23) circle (\vr) node [anchor = south] {0};
	\draw [fill=white] (-3,3.5) circle (\vr) node [anchor = north] {1};
	\draw [fill=white] (-1,3.5) circle (\vr) node [anchor = north] {2};
	\draw (-1.9,1.5) node {$G_1=C_3$};
	\draw (4,6.89)--(6,6.89); 
	\draw (7.41,5.41)--(7.41,3.41); %
	\draw (6,2)--(4,2); %
	\draw (2.59,3.41)--(2.59,5.41);
	\draw (2.07,7.35)--(4,6.89)--(2.59,5.41)--cycle; 
	\draw [fill=black] (4,6.89) circle (\vr) node [anchor = north] {00};
	\draw [fill=white] (2.59,5.41) circle (\vr) node [anchor = west] {01};
	\draw [fill=white] (2.07,7.35) circle (\vr) node [anchor = south] {02};
	\draw (2.07,1.48)--(2.59,3.41)--(4,2)--cycle; 
	\draw [fill=white] (2.59,3.41) circle (\vr) node [anchor = west] {10};
	\draw [fill=white] (4,2) circle (\vr) node [anchor = south] {12};
	\draw [fill=black] (2.07,1.48) circle (\vr) node [anchor = north] {11};
	\draw (7.93,1.48)--(6,2)--(7.41,3.41)--cycle; 
	\draw [fill=black] (7.93,1.48) circle (\vr) node [anchor = north] {22};
	\draw [fill=white] (6,2) circle (\vr) node [anchor = south] {21};
	\draw [fill=white] (7.41,3.41) circle (\vr)  node [anchor = east] {20};
	\draw (7.93,7.35)--(7.41,5.41)--(6,6.89)-- cycle; 
	\draw [fill=black] (6,6.89) circle (\vr) node [anchor = north] {30};
	\draw [fill=white] (7.41,5.41) circle (\vr) node [anchor = east] {32};
	\draw [fill=white] (7.93,7.35) circle (\vr)  node [anchor = south] {31};
	\draw (5,0.8) node {$K$};
	\end{tikzpicture}
\end{center}

\vspace{\baselineskip}

\begin{center}
	\begin{tikzpicture}[scale=0.50,style=thick]
	\def\vr{3.0pt/0.6}
	\begin{large}
		\draw (-4.8,4.4) node {$\otimes_{\varphi \circ f'}$};
		\draw (-3.,4.45) node {$K$};
		\draw (-1.8,4.4) node {$=$};
	\end{large}
	\draw (-7.5,5.23)--(-8.5,3.5);
	\draw (-6.5,3.5)--(-8.5,3.5);
	\draw (-7.5,5.23)--(-6.5,3.5);
	\draw [fill=white] (-7.5,5.23) circle (\vr) node [anchor = south] {0};
	\draw [fill=white] (-8.5,3.5) circle (\vr) node [anchor = north] {1};
	\draw [fill=white] (-6.5,3.5) circle (\vr) node [anchor = north] {2};
	\draw (-7.4,1.5) node {$G_3=C_3$};
	\draw (7.93,1.48)-- (9.47,-2.63);
	\draw (-1.46,-2.64)-- (2.07,1.48);
	\draw (7.53,-7.98)-- (2.47,-7.98); 
	\draw (4,6.83)--(6,6.83); 
	\draw (7.41,5.41)--(7.41,3.41); 
	\draw (6,2)--(4,2); 
	\draw (2.59,3.41)--(2.59,5.41); 
	\draw (2.07,7.35)--(4,6.83)--(2.59,5.41)--cycle; 
	\draw [fill=black] (4,6.83) circle (\vr) node [anchor = south] {000};
	\draw [fill=white] (2.59,5.41) circle (\vr);
	\draw [fill=white] (2.07,7.35) circle (\vr);
	\draw (2.07,1.48)--(2.59,3.41)--(4,2)--cycle; 
	\draw [fill=white] (4,2) circle (\vr);
	\draw [fill=white] (2.59,3.41) circle (\vr);
	\draw [fill=white] (2.07,1.48) circle (\vr) node [anchor = east] {$011$};
	\draw (7.93,1.48)--(6,2)--(7.41,3.41)--cycle; 
	\draw [fill=white] (6,2) circle (\vr);
	\draw [fill=white] (7.41,3.41) circle (\vr);
	\draw [fill=white] (7.93,1.48) circle (\vr) node [anchor = west] {$022$};
	\draw (7.93,7.35)--(7.41,5.41)--(6,6.83)--cycle; 
	\draw [fill=white] (7.41,5.41) circle (\vr);
	\draw [fill=white] (6,6.83) circle (\vr);
	\draw [fill=white] (7.93,7.35) circle (\vr);
	\draw (1.95,-4.05)--(1.95,-6.05); 
	\draw (0.54,-7.46)--(-1.46,-7.46); 
	\draw (-2.87,-6.05)--(-2.87,-4.05); 
	\draw (0.54,-2.64)--(-1.46,-2.64); 
	\draw (-1.46,-2.64)--(-2.87,-4.05)--(-3.39,-2.12)--cycle; 
	\draw [fill=white] (-3.39,-2.12) circle (\vr);
	\draw [fill=white] (-1.46,-2.64) circle (\vr); \draw (-1.1,-2.64) node [anchor = north] {$100$};
	\draw [fill=white] (-2.87,-4.05) circle (\vr);
	
	\draw (-2.87,-6.05)-- (-3.39,-7.98)--(-1.46,-7.46)--cycle; 
	\draw [fill=black] (-3.39,-7.98) circle (\vr) node [anchor = north] {$111$};
	\draw [fill=white] (-1.46,-7.46) circle (\vr);
	\draw [fill=white] (-2.87,-6.05) circle (\vr);
	
	\draw (2.47,-7.98)--(1.95,-6.05)--(0.54,-7.46)--cycle; 
	\draw [fill=white] (2.47,-7.98) circle (\vr) node [anchor = north] {$122$};
	\draw [fill=white] (1.95,-6.05) circle (\vr);
	\draw [fill=white] (0.54,-7.46) circle (\vr);
	
	\draw (0.54,-2.64)--(2.47,-2.12)--(1.95,-4.05)--cycle; 
	\draw [fill=white] (0.54,-2.64) circle (\vr);
	\draw [fill=white] (2.47,-2.12) circle (\vr);
	\draw [fill=white] (1.95,-4.05) circle (\vr);
	\draw (9.47,-2.63)--(11.47,-2.63); 
	\draw (12.88,-4.05)--(12.88,-6.05); 
	\draw (11.47,-7.46)--(9.47,-7.46); 
	\draw (8.05,-6.05)--(8.05,-4.05); 
	\draw (9.47,-2.63)--(8.05,-4.05)--(7.53,-2.11)--cycle; 
	\draw [fill=white] (7.53,-2.11) circle (\vr);
	\draw [fill=white] (9.47,-2.63) circle (\vr); \draw (9.85,-2.63) node [anchor = north] {$200$};
	\draw [fill=white] (8.05,-4.05) circle (\vr);	
	\draw (8.05,-6.05)--(7.53,-7.98)--(9.47,-7.46)--cycle; 
	\draw [fill=white] (7.53,-7.98) circle (\vr) node [anchor = north] {$211$};
	\draw [fill=white] (9.47,-7.46) circle (\vr);
	\draw [fill=white] (8.05,-6.05) circle (\vr);
	\draw (11.47,-7.46)--(12.88,-6.05)--(13.4,-7.98)--cycle; 
	\draw [fill=black] (13.4,-7.98) circle (\vr) node [anchor = north] {$222$};
	\draw [fill=white] (12.88,-6.05) circle (\vr);
	\draw [fill=white] (11.47,-7.46) circle (\vr);
	\draw (12.88,-4.05)--(11.47,-2.63)--(13.4,-2.11)--cycle; 
	\draw [fill=white] (13.4,-2.11) circle (\vr);
	\draw [fill=white] (11.47,-2.63) circle (\vr);
	\draw [fill=white] (12.88,-4.05) circle (\vr);
	\end{tikzpicture}
\end{center}

\caption{Construction of graph $C_3\otimes_{f'} C_4\otimes_{f} C_3$, where $f : i \mapsto i \ (\bmod\ 3)$ and $f' = {\rm id}$.}
\label{fig:more_factors}
\end{figure}

It is now easy to see that Sierpi\'nski products possess a nice recursive structure,
similar to Sierpi\'nski graphs and generalized Sierpi\'nski graphs. 
By the same reasoning as above, the product 
$G_{m}\otimes_{f_{m-1}} \cdots \otimes_{f_{2}} G_{2} \otimes_{f_{1}} G_{1}$, 
where $V(G_\ell) = \{0,1,\ldots, |G_\ell|-1\}$, 
and $f_\ell: V(G_{\ell+1}) \to V(G_\ell)$, $\ell = 1,\ldots,m-1$, are arbitrary functions,
can be constructed as follows.

\begin{itemize}
\item First, take $|G_2|$ copies of the graph $G_{1}$ 
and label them $iG_{1}$, $i\in\{0,\ldots,|G_2|-1\}$. 
Vertices of these graphs have labels $g_{2}g_{1}$.

\item Connect any two copies $iG_1$ and $jG_1$ if there is an edge $\{i,j\}$ in $G_2$. 
More precisely, if $\{i,j\}\in E(G_2)$, we add an edge 
$\{if_1(j),jf_1(i)\}$ between $iG_1$ and $jG_1$.
The resulting graph is then indeed the Sierpi\'nski product $G_2 \otimes_{f_1} G_1$ 
and the corresponding function $\varphi_1: V(G_2) \rightarrow V(G_2 \otimes G_1)$ 
maps $i$ to $if_1(i)$ for every $i\in \{0,\ldots,|G_2|-1\}$.
  
\item Next we form the Sierpi\'nski product of graphs $G_3$ and $K(2) := G_2\otimes_{f_1} G_1$.
To do so we take $|G_3|$ copies of graph $K(2) $, label them $iK(2) $, $i\in\{0,\ldots,|G_3|-1\}$, 
and connect $iK(2)$ and $jK(2)$ whenever $\{i,j\}$ is an edge in $G_3$. 
Such an edge then has the form $\{if_2(j)f_1(f_2(j)),jf_2(i)f_1(f_2(i))\}$.
 
\item The final step is to form the Sierpi\'nski product of graphs $G_m$ and $K(m-1)$ 
in the same way as we formed all the products so far: 
make $|G_m|$ copies of $K(m-1)$ and label them $iK(m-1)$;
then for every edge $\{i,j\}$ in $G_m$ we add an edge between copies $iK(m-1)$ and $jK(m-1)$.
Such an edge has then the following form
$$\{if_{m-1}(j)\ldots f_1(f_2\dots (f_{m-1}(j))\dots)\, ,\ 
jf_{m-1}(i)\ldots f_1(f_2\dots (f_{m-1}(i))\dots)\}.$$
 
The resulting graph is the product 
$G_{m}\otimes_{f_{m-1}} \cdots \otimes_{f_{2}} G_{2} \otimes_{f_{1}} G_{1}$.
\end{itemize}

If $G_1=\dots=G_m=G$ and functions $f_1,\ldots,f_{m-1}$ are all the identity function,
then $G_{m}\otimes_{f_{m-1}} \cdots \otimes_{f_{2}} G_{2} \otimes_{f_{1}} G_{1}$ is the
generalized Sierpi\'nski graph $S_G^n$; see also \cite{Kovse}.\\

We can calculate the order and the size of the Sierpi\'nski product of multiple
factors directly from the above construction.

\begin{proposition}
\label{prop:edges}
Let $m\ge 2$,  and let $ G_1 ,\ldots ,G_m $ be arbitrary graphs.  
Further let $f_{1}: V(G_2) \rightarrow V(G_{1}),\dots$, $f_{m-1}: V(G_m) \rightarrow V(G_{m-1})$ 
be arbitrary functions.
Then the order and size of Sierpi\'nski product $G_m \otimes_{f_{m-1}} \cdots \otimes_{f_1} G_1$ are as follows
\begin{align*}
| G_m \otimes_{f_{m-1}} \cdots \otimes_{f_1} G_1 | 
		&= \prod_{\ell=1}^{m} |G_\ell | \,, \\
|| G_m \otimes_{f_{m-1}} \cdots \otimes_{f_1} G_1 || 
		&= \sum_{\ell=1}^{m} \left( \prod_{j=\ell+1}^{m} |G_j| \right) ||G_\ell || \,. 
\end{align*}
\end{proposition}

%

Note that neither the order nor the size of the Sierpi\'nski product depends on the functions $f_\ell$.

If $K := G_m \otimes_{f_{m-1}} \cdots \otimes_{f_1} G_1$, $m\ge 2$, is a Sierpi\'nski product, then the vertices of $K$ with some common prefix $g_m\ldots g_{\ell+1}$ ($\ell \ge 0$) belong to the same copy of $H := G_{\ell} \otimes_{f_{\ell-1}} \cdots \otimes_{f_1} G_1$. We generalize the notation from Section 2 and denote such copy by $g_m\ldots g_{\ell} H$, where $H = G_{\ell} \otimes_{f_{\ell-1}} \cdots \otimes_{f_1} G_1$. We will use this notation to state an upper bound on the diameter of a Sierpi\'nski product. 
But first let us prove an auxiliary result which we will require in the result about the diameter.

\begin{lemma}
	\label{lem:recursion}
	Let $\{a_m\}_{m\in\mathbb{N}}$ and $\{d_m\}_{m\in\mathbb{N}}$ be integer sequences 
	satisfying the following recursion
		$$ a_m = \left( d_{m} +1 \right) a_{m-1} + d_{m} \,, \quad a_1 = d_1 \,.$$
	Denote $[m] := \{1,\ldots,m\}$.
	Then the closed formula for $\{a_m\}_{m\in\mathbb{N}}$ is given by 
		$$ a_m = \sum_{\ell=1}^{m} \sum_{ \{i_1,\ldots,i_\ell\} \subseteq [m]} d_{i_1} \cdots d_{i_\ell} \,.$$
\end{lemma}

\begin{proof}
If $m=1$, the closed formula above gives us $a_1 = \sum_{\ell=1}^{1} \sum_{ \{i_\ell\} \subseteq \{1\}} d_j = d_1$. For $m>1$ we have
\begin{align*}
	a_m &= \left( d_{m} +1 \right) a_{m-1} + d_{m} 
		= \left( d_m+1\right) \left( \sum_{\ell=1}^{m-1} \sum_{ \{i_1,\ldots,i_\ell\} 
				\subseteq [m-1]} d_{i_1} \cdots d_{i_\ell} \right) + d_{m}\\
		&=	\sum_{\ell=1}^{m-1} \sum_{ \{i_1,\ldots,i_\ell\} 
					\subseteq [m-1]} d_{i_1} \cdots d_{i_\ell} \cdot d_m  + 
				\sum_{\ell=1}^{m-1} \sum_{ \{i_1,\ldots,i_\ell\} 
					\subseteq [m-1]} d_{i_1} \cdots d_{i_\ell}
				+ d_{m}\\
		&=	\sum_{\ell=1}^{m} \sum_{ \{i_1,\ldots,i_\ell\} \subseteq [m]} d_{i_1} \cdots d_{i_\ell} 
	\,,
\end{align*}
which completes the proof.
\end{proof}

When dealing with distances it is also useful to note that the following observation holds.
\begin{proposition}
	\label{prop:distances}
	Let $m\ge 2$,  and let $ G_1 ,\ldots ,G_m $ be arbitrary graphs.  
	Let $f_{1}: V(G_2) \rightarrow V(G_{1})$, $\ldots$, $f_{m-1}: V(G_m) \rightarrow V(G_{m-1})$ 
	be arbitrary functions.
	Finally, let $g$, $g'$ be vertices in $G_m \otimes_{f_{m-1}} \cdots \otimes_{f_1} G_1$ 
	and let $\ell \in \{1,\ldots,m\}$ be the greatest index of coordinates in which $g$ and $g'$ 
	differ, i.e., $g = (g_m,\ldots,g_{\ell+1},g_\ell,\ldots,g_1) =: \underline{g}g_\ell\ldots g_1$ 
	and $g' = (g_m,\ldots,g_{\ell+1},g_\ell',\ldots, g_1') =: \underline{g} g_\ell'\ldots g_1'$. 
	Then $g$ and $g'$ belong to the same copy of 
	$H := G_{\ell} \otimes_{f_{\ell-1}} \cdots \otimes_{f_1} G_1$, and 
		$$d_G(g,g') = d_H(g_\ell \ldots  g_1,g_\ell'\ldots g_1') \,.$$
\end{proposition}

\begin{proposition}
	\label{prop:diameter}
	Let $m\ge 2$,  and let $ G_1 ,\ldots ,G_m $ be arbitrary graphs.  
	Further let $f_{1}: V(G_2) \rightarrow V(G_{1})$, $\ldots$, 
	$f_{m-1}: V(G_m) \rightarrow V(G_{m-1})$ be arbitrary functions.
	Then
		$$ 	
		\text{{\em diam}} \left( G_m \otimes_{f_{m-1}} \cdots \otimes_{f_1} G_1 \right) \leq 
		\sum_{\ell=1}^{m} \sum_{ \{i_1,\ldots,i_\ell\} \subseteq [m]} 
			\text{{\em diam}}\left(G_{i_1}\right) \cdots \text{{\em diam}}\left(G_{i_\ell} \right) 	\,.
		$$
\end{proposition}

\begin{proof} 
Denote $G = G_m \otimes_{f_{m-1}} \cdots \otimes_{f_1} G_1 $ and let 
$g = g_m \ldots g_1$, $g' = g_m' \ldots g_1'$ be vertices of $G$.
Due to Proposition~\ref{prop:distances} let us assume that $g$ and $g'$ differ in the $m$-th 
coordinate. Because the vertices are in different copies of 
$H := G_{m-1} \otimes_{f_{m-2}} \cdots \otimes_{f_1} G_1$, we have to find a shortest 
path from $g_mH$ to $g_m'H$, and such path has length at most 
\begin{equation}
	\label{eq:diam}
	\text{diam} (H) (\text{diam} (G_{m}) + 1 ) +\text{diam} (G_{m}) \,,
\end{equation}
because in worst case we have to cross both graphs $g_mH$ and $g_m'H$, but also some 
other copies isomorphic to $H$. Note that on such shortest path we cannot cross more 
subgraphs isomorphic to $H$ than $\text{diam} (G_{m}) + 1$, otherwise we would have a path 
in $G_m$ that is longer than its diameter $\text{diam} (G_{m})$. Every time we cross between 
the subgraphs we add another edge to our shortest path, and this happens in at most 
$\text{diam} (G_{m})$ cases.

The result now follows from the fact that \eqref{eq:diam} satisfies the recursion in 
Lemma~\ref{lem:recursion} with 
$a_m := \text{diam} \left( G_m \otimes_{f_{m-1}} \cdots \otimes_{f_1} G_1 \right)$, and 
$d_m := \text{diam} \left( G_m \right)$ for $m\in\mathbb{N}$.
\end{proof}


In the next examples the above bound is tight.
\begin{example}
\label{ex:paths}
Let $n,m\ge 2$ and $f: V(P_n)=\{1,\ldots,n\} \rightarrow V(P_m)=\{1,\ldots,m\}$ be defined as
$$
f(i) = \left\lbrace \begin{array}{ll}
1\ & \ \mbox{if} \ i \equiv 1,2 \mod 4,\\
m\ & \ \mbox{if} \  i \equiv 0,3 \mod 4.
\end{array}
\right.
$$
Then, diam$(P_n\otimes_f P_m) = nm -1 = $ diam $(P_n)$ diam $(P_m) + $ diam $(P_n) +$ diam $(P_m)$, which equals the upper bound from Proposition~\ref{prop:diameter} (cf. Figure~\ref{fig:examplePaths}).

\begin{figure}[htb!]
\centering
\begin{tikzpicture}[scale = 0.58, style = thick]
\def\vr{3pt}
	\draw (0,4)--(0.5,2)--(1,0)--(2.5,0)--(3,2)--(3.5,4)--(5,4)--(5.5,2)
		--(6,0)--(7.5,0)--(8,2)--(8.5,4)--(10,4)--(10.5,2)--(11,0)--(12.5,0)
		--(13,2)--(13.5,4)--(15,4)--(15.5,2)--(16,0)--(17.5,0)--(18,2)--(18.5,4)
		--(20,4)--(20.5,2)--(21,0)--(22.5,0)--(23,2)--(23.5,4);
\begin{small}
	\draw [fill=white] (0,4) circle (\vr) node [anchor = south east] {16};
	\draw [fill=white] (0.5,2) circle (\vr) node [anchor = east] {15};
	\draw [fill=white] (1,0) circle (\vr) node [anchor = north east] {14};
	\draw [fill=white] (2.5,0) circle (\vr) node [anchor = north west] {13};
	\draw [fill=white] (3,2) circle (\vr) node [anchor = west] {12};
	\draw [fill=white] (3.5,4) circle (\vr) node [anchor = south east] {11};
	\draw [fill=white] (5,4) circle (\vr) node [anchor = south west] {21};
	\draw [fill=white] (5.5,2) circle (\vr) node [anchor = east] {22};
	\draw [fill=white] (6,0) circle (\vr) node [anchor = north east] {23};
	\draw [fill=white] (7.5,0) circle (\vr) node [anchor = north west] {24};
	\draw [fill=white] (8,2) circle (\vr) node [anchor = west] {25};
	\draw [fill=white] (8.5,4) circle (\vr) node [anchor = south east] {26};
	\draw [fill=white] (10,4) circle (\vr) node [anchor = south west] {31};
	\draw [fill=white] (10.5,2) circle (\vr) node [anchor = east] {32};
	\draw [fill=white] (11,0) circle (\vr) node [anchor = north east] {33};
	\draw [fill=white] (12.5,0) circle (\vr) node [anchor = north west] {34};
	\draw [fill=white] (13,2) circle (\vr) node [anchor = west] {35};
	\draw [fill=white] (13.5,4) circle (\vr) node [anchor = south east] {36};
	\draw [fill=white] (15,4) circle (\vr) node [anchor = south west] {46};
	\draw [fill=white] (15.5,2) circle (\vr) node [anchor = east] {45};
	\draw [fill=white] (16,0) circle (\vr) node [anchor = north east] {44};
	\draw [fill=white] (17.5,0) circle (\vr) node [anchor = north west] {43};
	\draw [fill=white] (18,2) circle (\vr) node [anchor = west] {42};
	\draw [fill=white] (18.5,4) circle (\vr) node [anchor = south east] {41};
	\draw [fill=white] (20,4) circle (\vr) node [anchor = south west] {56};
	\draw [fill=white] (20.5,2) circle (\vr) node [anchor = east] {55};
	\draw [fill=white] (21,0) circle (\vr) node [anchor = north east] {54};
	\draw [fill=white] (22.5,0) circle (\vr) node [anchor = north west] {53};
	\draw [fill=white] (23,2) circle (\vr) node [anchor = west] {52};
	\draw [fill=white] (23.5,4) circle (\vr) node [anchor = south west] {51};
\end{small}
\end{tikzpicture}
\caption{A visualization of Example~\ref{ex:paths} for $n=5$ and $m=6$.}
\label{fig:examplePaths}
\end{figure}
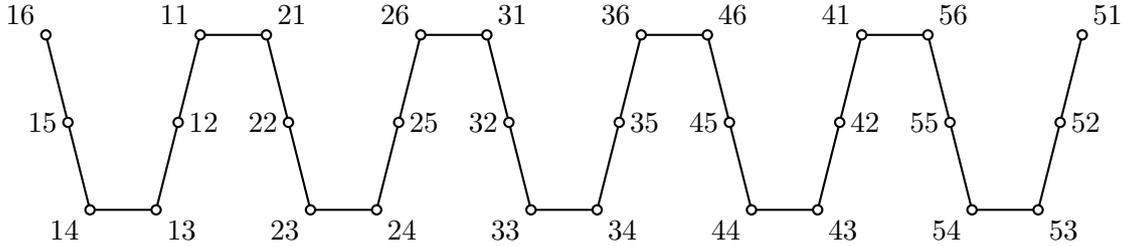
\end{example}

\begin{example}
Let $p\ge 2$, $n\ge 1$. Then $K_p \otimes \cdots \otimes K_p = S_p^n$.
It is a well-known fact that diam$(K_p) = 1$, and it is also known (cf. \cite[Proposition 2.12]{HKZ2017}) that diam$\left( S_p^n \right) = 2^n-1$. The upper bound in Proposition~\ref{prop:diameter} for the case of $S_p^n$ equals $\sum_{\ell=1}^{n} \sum_{ \{i_1,\ldots,i_\ell\} \subseteq [n]} 1$, which is just the number of non-empty subsets of $[n] = \{1,\ldots,n\}$ and there are exactly $2^n-1$ such subsets, so the bound is tight for $S_p^n$.
\end{example}

Distances are also important for the following open problem.

\begin{problem}
 Suppose all graphs $G_{\ell}$ are connected. 
 What can we say about the girth of $G_m \otimes_{f_{m-1}} \cdots \otimes_{f_1} G_1$ ? 
 How is it related to the girths of its factors?
\end{problem}

\section{Conclusion}

This paper generalizes Sierpi\'nski graphs even further than generalized Sierpi\'nski graphs, where the whole structure is based only on one graph. Here we create a product like structure of two (or more) factors. Some basic graph theoretical properties are studied in detail, and planar Sierpi\'nski products are completely characterized. Apart from this the symmetries of Sierpi\'nski products are studied as well. In general, these are not fully understood. In many cases we are able to determine the automorphism group of Sierpi\'nski product of two graphs exactly. 


In \cite{ImrichPeterin} an algorithm is given for recognizing generalized Sierpi\'nski graphs.
Given a graph it is also natural to ask whether it can be represented as a Sierpi\'nski product of two or more graphs. Moreover, one can ask if such a representation is unique. 
The latter question has a negative answer.
Consider the Sierpi\'nski product of $C_4$ and $2K_3+e$ with function $f$ as in Figure \ref{fig:counterexample}. 
It can be easily verified that it is isomorphic to $C_8 \otimes_{f'} C_3$ where $f':V(C_8) \to V(C_3)$ 
is defined by $f'(1)=f'(2)=f'(5)=f'(6)=1$ and
$f'(3)=f'(4)=f'(7)=f'(8)=2$. However, in this case not all the factors are prime 
with respect to the Sierpi\'nski product: $C_8$ can be represented as a Sierpi\'nski product
of $C_4$ and $K_2$ while $2K_3+e$ can be represented as a Sierpi\'nski product
of $K_2$ and $C_3$. It would be interesting to see whether there exist prime graphs
with respect to the Sierpi\'nski product  $G,H,G',H'$  and functions $f:V(G) \to V(H)$,
$f':V(G') \to V(H')$ such that $G,H$ are not isomorphic to $G',H'$ while
$G \otimes_f H$ is isomorphic to $G' \otimes_{f'} H'$.

\section*{Acknowledgements}

The authors would like to thank Wilfried Imrich for fruitful discussion.
This work was supported in part by
`Agencija za raziskovalno dejavnost Republike Slovenije' via 
Grants  P1--0294, J1--9187, J1--7051, J1--7110 and N1--0032.

\newpage
\noindent
{\sc Jurij Kovi\v{c}}  \\
Institute of Mathematics, Physics, and Mechanics\\
Jadranska 19, 1000 Ljubljana, Slovenia\\
and \\
University of Primorska, FAMNIT \\
Glagolja\v ska 8, 6000 Koper, Slovenia\\
\texttt{jurij.kovic@siol.net}\\

\noindent
{\sc Toma\v z Pisanski} \\ 
University of Primorska, FAMNIT \\
Glagolja\v ska 8, 6000 Koper, Slovenia\\
and\\
Institute of Mathematics, Physics, and Mechanics\\
Jadranska 19, 1000 Ljubljana, Slovenia\\
\texttt{tomaz.pisanski@upr.si}\\

\noindent
{\sc Sara Sabrina Zemlji\v c}\\ 
Comenius University, Bratislava, Slovakia
\\ 
and\\
Institute of Mathematics, Physics, and Mechanics\\
Jadranska 19, 1000 Ljubljana, Slovenia\\
\texttt{sara.zemljic@gmail.com}\\

\noindent
{\sc Arjana \v Zitnik}\\ 
University of Ljubljana, Faculty of Mathematics and Physics\\
Jadranska 19, 1000 Ljubljana, Slovenia\\ 
and\\
Institute of Mathematics, Physics, and Mechanics\\
Jadranska 19, 1000 Ljubljana, Slovenia\\
\texttt{arjana.zitnik@fmf.uni-lj.si}\\ 

\end{document}